\documentclass[a4paper]{amsart}
\usepackage{amsthm}
\usepackage{amssymb}
\usepackage{amsmath}
\usepackage{verbatim,enumerate}

\renewcommand{\thefootnote}{\fnsymbol{footnote}}

\newtheorem{theorem}{Theorem}[section]
\newtheorem{lemma}[theorem]{Lemma}

\newtheorem{proposition}[theorem]{Proposition}
\newtheorem{remark}[theorem]{Remark}

\newtheorem{example}[theorem]{Example}

\newtheorem*{example*}{Example}

\newtheorem*{remark*}{Remark}

\newtheorem{corollary}[theorem]{Corollary}
\newtheorem*{corollary*}{Corollary}

\newtheorem{definition}[theorem]{Definition}

\newtheorem*{definition*}{Definition}

\newtheorem*{notation*}{Notation}

\newtheorem{notation}[theorem]{Notation}

\numberwithin{equation}{section}

\gdef\myletter{}

\let\savetheequation\theequation
\def\theequation{\savetheequation\myletter}

\def\v{{\bf v}}
\def\w{\mathbf{w}}
\def\calV{\mathcal W}
\def\Id{\mathrm{I}_d}
\def\reg{\mathrm{reg}}

\def\dss{\displaystyle}

\newcommand{\CC}{{\mathbb C}}

\newcommand{\RR}{{\mathbb R}}
\newcommand{\ZZ}{{\mathbb Z}}

\newcommand{\PP}{{\mathbb P}}

\newcommand{\NN}{{\mathbb N}}

\newcommand{\lt}{\textsc{lt}}

\newcommand{\Van}{{\mathrm{Van}}}

\def \I{\mathbf{I}}

\def \hat{\widehat}

\def \b0{{\bf 0}}

\def \hsta{\,\hat*\,}

\long\def\symbolfootnote[#1]#2{\begingroup%
\def\thefootnote{\fnsymbol{footnote}}\footnote[#1]{#2}\endgroup}

\def\lI{\langle\lt(I)\rangle}

\begin{document}

\title[Chebyshev constants and transfinite diameter on curves]{Chebyshev constants, transfinite diameter, and computation on complex algebraic curves}

\author{W. Baleikorocau and S. Ma`u}



\address{}
\email{wbaleikorocau@gmail.com}

\address{University of Auckland, Auckland, New Zealand}
\email{s.mau@auckland.ac.nz}

\date{\today}

\begin{abstract}
New notions of directional Chebyshev constant and transfinite diameter    
have recently been studied on certain algebraic curves in $\mathbb{C}^2$ \cite{mau:chebyshev}.  The theory is extended here to curves in $\CC^N$ for arbitrary $N$.  The results are analogous but require more   methods from computational algebraic geometry.  
\end{abstract}

\maketitle

\section{Introduction}

The goal of this paper is to study a notion of transfinite diameter on algebraic curves in $\CC^N$ ($N>1$).  
This will be a natural generalization of the Fekete-Leja transfinite diameter of a compact set in $\CC^N$.  The importance of the latter has increased in recent years as its geometric and analytic aspects have become better understood (see e.g. \cite{rumely:robin}, \cite{bermanboucksom:growth}, \cite{bloomlev:transfinite}).
 
We briefly recall the definition of the Fekete-Leja transfinite diameter.

Let $\{z^{\alpha_j}\}_{j=1}^{\infty}$ be the monomials in $N$ variables listed according to a \emph{graded order} (i.e., $|\alpha_n|\leq|\alpha_k|$ whenever $n<k$).   Here we are using  standard multi-index notation: if $\alpha_j=(\alpha_{j1},...,\alpha_{jN})$, then $z^{\alpha_j}=z_1^{\alpha_{j1}}z_2^{\alpha_{j2}}\cdots z_N^{\alpha_{jN}}$ and  $|\alpha_j|=\alpha_{j1}+\cdots+\alpha_{jN}$ denotes the total degree.    
Given a positive integer $M$ and points  $\{\zeta_1,...,\zeta_M\}\subset\CC^N$, the $M\times M$ determinant
\begin{equation}\label{eqn:0.1}
\Van(\zeta_1,...,\zeta_M) \ = \ 
\det
\begin{pmatrix}
1 & 1 & \cdots & 1 \\
z^{\alpha_2}(\zeta_1) & z^{\alpha_2}(\zeta_2) & \cdots & z^{\alpha_2}(\zeta_M) \\
\vdots & \vdots & \ddots & \vdots \\
z^{\alpha_M}(\zeta_1) & z^{\alpha_M}(\zeta_2) & \cdots & z^{\alpha_M}(\zeta_M)
\end{pmatrix}
\end{equation}
is called a \emph{Vandermonde determinant} of order $M$.\footnote{Note that $z^{\alpha_1}=1$ since $\alpha_1=(0,0,...,0)$.}  

Given a compact set $K\subset\CC^N$, the $n$-th order diameter $d_n(K)$ is defined as follows.  Let $m_n$ be the number of monomials of degree at most $n$ in $N$ variables, and let $l_n=\sum_{j=1}^{m_n} |\alpha_j|$ be the sum of the degrees. Then 
\begin{equation}\label{eqn:0.2}
d_n(K) \ = \ \sup\{|\Van(\zeta_1,...,\zeta_{m_n})|^{\frac{1}{l_n}}: \{\zeta_1,...,\zeta_{m_n}\}\subset K\}.
\end{equation}
The Fekete-Leja transfinite diameter of $K$ is then given by \begin{equation}\label{eqn:0.3} d(K):=\lim_{n\to\infty}d_n(K).\end{equation}  
The existence of the limit on the right-hand side of (\ref{eqn:0.2}) was verified by Fekete \cite{fekete:uber} when $N=1$ and by Zaharjuta \cite{zaharjuta:transfinite} in general.  Recent studies of the transfinite diameter and related notions are \cite{bloomlev:transfinite} and the survey \cite{Zaharjuta:transfinite2012}.

Now consider an algebraic variety $V\subset\CC^N$.   The polynomials restricted to $V$ form the \emph{coordinate ring} $\CC[V]$; $p=q$ in $\CC[V]$ means that $p(z)=q(z)$ for all $z\in V$.  The monomials in $N$ variables 
span $\CC[V]$ (as a complex vector space), and linear dependencies among the monomials, induced by restricting to $V$, are given by the ideal $$\I(V)=\{p\in\CC[z]:p(a)=0 \hbox{ for all } a\in V\}.$$ 

One can systematically reduce the set of monomials to a basis of $\CC[V]$ as follows. Let  $\{z^{\alpha_j}\}_{j=1}^{\infty}$ denote the monomials indexed according to a graded ordering.  
One then goes through the collection, removing linearly dependent monomials as they arise.  For example, remove  $z^{\alpha_{m+1}}$ if it is linearly dependent with respect to   $\{z^{\alpha_j}\}_{j=1}^m$. 

\def\calB{\mathcal{B}}

Let now $\calB=\{e_j\}_{j=1}^{\infty}$ denote the reduced set of monomials, which is a basis for $\CC[V]$ by definition.    For a positive integer $M$ and points $\{\zeta_1,...,\zeta_M\}\subset V$, define the $M\times M$ Vandermonde determinant   $\Van_{\calB}(\zeta_1,...,\zeta_M)$ to be given by (\ref{eqn:0.1}) with each occurence of $z^{\alpha_j}$ replaced by $e_j$.   

If we let now $m_n=m_n(V)$ be the number of monomials in $\calB$ of degree at most $n$, and $l_n=l_n(V):=\sum_{j=1}^{m_n} |\alpha_j|$, then equation (\ref{eqn:0.2}) defines the $n$-th order diameter of a compact set $K\subset V$ with $\Van_{\calB}(\cdot)$ replacing $\Van(\cdot)$.  Finally, define the \emph{transfinite diameter} by
\begin{equation} \label{eqn:0.4}
d(K) = \limsup_{n\to\infty} d_n(K).
\end{equation}

  The main theorem (Theorem \ref{thm:d=t}) says that when $V$ is an algebraic curve (satisfying some  additional technical properties), the lim sup in (\ref{eqn:0.4}) may be replaced by a limit as in (\ref{eqn:0.3}).  When $N=2$, this was done in \cite{mau:chebyshev}.  The point of this paper is that the methods there generalize in a natural way to arbitrary $N$.  The main idea, following \cite{zaharjuta:transfinite}, is to relate the transfinite diameter to various \emph{Chebyshev constants}, whose limiting properties can be proved directly.   

The paper is organized as follows.  Section 2 is devoted to recalling background material in computational algebraic geometry and in describing the notation that will be used in the rest of the paper.  As indicated above, systematic computation requires an ordering on monomials.    In several variables there are several ways to  order monomials that respect degree (such orderings are called \emph{graded orderings}); we will work exclusively with the \emph{grevlex ordering} (see Section \ref{subsection:groebner}).

In Section 3, we relate computation on an algebraic curve $V\subset\CC^N$ to its geometry.  
  To study this relationship it is convenient to view $V$  projectively, i.e., consider $V\subset\CC\PP^N=\CC^N\cup H_{\infty}$  (where $H_{\infty}$ denotes the hyperplane at infinity); $V$ extends continuously as a projective curve across $H_{\infty}$.  Under mild restrictions on points of $H_{\infty}\cap V$, algebraic computation in $\CC[V]$ has some nice properties.  This section builds on preliminary investigations in $\CC^3$ carried out in  \cite{baleikorocau:groebner}.

In section 4, we study Chebyshev constants.  Following an idea in \cite{bloom:families}, we define Chebyshev constants associated to homogeneous polynomials.  This includes the directional Chebyshev constants of  \cite{mau:chebyshev} as special cases.  We then derive geometric properties of these Chebyshev constants.

In section 5, we prove Theorem {\ref{thm:d=t}}.  The theorem relates the notions of transfinite diameter and directional Chebyshev constant, and shows that the transfinite diameter is given by a well-defined limiting process.    Further properties of transfinite diameter are also shown.  

The main results of sections 4 and 5 are the same as those of \cite{mau:chebyshev}, with some arguments  simplified.  In particular, most properties of the transfinite diameter given here are not proved directly but follow immediately from corresponding properties of directional Chebyshev constants, which are studied here in more detail.  The directional Chebyshev constant is the more primitive notion and its properties are easier to derive.  We remark that in a more general setting, the part of Theorem \ref{thm:d=t} dealing with existence of the limit can be proved using Bernstein-Markov measures rather than Chebyshev constants \cite{bermanboucksom:growth}.

We close the paper by illustrating the relationship between directional Chebyshev constants associated to $K\subset V$ and Robin constants associated to the Siciak-Zaharjuta extremal function of $K$, which is the maximal plurisubharmonic function given by 
$$
V_K(z):=\sup\{\frac{1}{\deg p}\log|p(z)|: p \hbox{ a polynomial with } \|p\|_K\leq 1\}
$$
(here $\|p\|_K=\sup_{z\in K}|p(z)|$ denotes the uniform norm).  This relationship will be studied further in a future paper.


\section{Preliminaries}

This section reviews essential background and notation we will need, with proofs omitted.  We refer to  \cite{coxlittleoshea:ideals}, whose notation we follow closely. 

\subsection{Dimension and nonsingularity}

Write $z=z_1,\ldots,z_N$ for the standard variables or coordinates on $\CC^N$, and write $\CC[z]=\CC[z_1,\ldots,z_N]$ for the ring of polynomials over $\CC$ in these variables.  We use standard multi-index notation: if $\alpha=(\alpha_1,\ldots,\alpha_N)$ is a multi-index then $z^{\alpha}=z_1^{\alpha_1}\cdots z_n^{\alpha_N}$ and  
$|\alpha|=\alpha_1+\cdots+\alpha_N$.

Given a (nonempty) algebraic variety $V=\{z\in\CC^N:P_1(z)=\cdots= P_k(z)=0\},$ where $P_1,...,P_k$ are polynomials in $\CC[z]$, let $$\I(V):=\{p\in\CC[z]:p(z)=0 \hbox{ for all } z\in V\}$$ 
be the ideal of $V$.  The polynomials restricted to $V$ can be identified with elements of the factor ring $\CC[z]/\I(V)=:\CC[V]$, called the \emph{coordinate ring of $V$}.

    Define the degree on $V$  of a polynomial  $p$ by $$\deg_V(p)=\min\{\deg(q): q(z)= p(z)\hbox{ for all } z\in V\},$$ where $\deg$ denotes the usual degree in $\CC[z]$, i.e., $\deg (c_{\alpha}z^{\alpha}):=|\alpha|$ ($c_{\alpha}\in\CC\setminus\{0\}$), and for any polynomials $p_1,p_2$, $\deg(p_1+p_2):=\max\{\deg p_1,\deg p_2\}$.  


 Next, for a non-negative integer $s$ write
$$
\CC[z]_{\leq s} := \{p\in\CC[z]: \deg(p)\leq s\} \hbox{ and } \CC[V]_{\leq s} := \{p\in\CC[V]: \deg_V(p)\leq s\}
$$
for the polynomials of degree at most $s$.  As a vector space over $\CC$ we have  $\dim(\CC[z]_{\leq s})=\binom{N+s}{s}$ as can be seen by counting the monomials of degree $\leq s$ in $z$, and $\dim(\CC[V]_{\leq s})\leq\dim(\CC[z]_{\leq s})$.  
It is well-known that for large $s$, $\dim(\CC[V]_{\leq s})$ is a polynomial in $s$, $H(s)$ (called the \emph{Hilbert polynomial of $V$}).  It is also a well-known fact that $\deg(H)=1$ (i.e., $H(s)=as+b$, $a\in\NN$, $b\in\ZZ$) if and only if at all but a finite number of exceptional points, $V$ is a \emph{complex manifold of dimension 1}.  For such a non-exceptional point $p\in V$ there is a local one-to-one holomorphic map $\varphi:D\to V$ (where $D=\{|t|<1\}\subset\CC$) with   $\varphi(0)=p$.  $V$ is then said to be an \emph{algebraic  curve}, and $p$ is a \emph{nonsingular point}.
We will work exclusively with algebraic curves in this paper.

We recall a useful characterization of nonsingular points.  Given a collection of polynomials $F=\{f_1,...,f_s\}$ and a point $p$, define the $s\times N$ matrix of partial derivatives 
$$
J_p(F)=J_p(f_1,...,f_s) \ := \ \begin{pmatrix}
\partial f_1/\partial z_1(p) &\cdots& \partial f_1/\partial z_N(p) \\
\vdots & \ddots & \vdots \\
\partial f_s/\partial z_1(p) &\cdots & \partial f_s/\partial z_N(p)
   \end{pmatrix}.
$$

\begin{proposition} \label{prop:2.1}
Let $V$ be an algebraic curve in $\CC^N$, and $\I(V)=\langle f_1,...,f_s\rangle$.  Then $s\geq N-1$ and $p\in V$ is nonsingular if and only if $J_p(f_1,...,f_s)$ has rank $N-1$. \qed
\end{proposition}

Hence the normal vectors $\{\nabla f_1(p),...,\nabla f_s(p)\}$ span a (complex) hyperplane $H$ where $\nabla f_i(p)=(\partial f_i/\partial z_1(p),...,\partial f_i/\partial z_N(p))$. The tangent space to $V$ at $p$ is then the orthogonal complement.

\subsection{Groebner bases and computation}\label{subsection:groebner}

The main tool to carry out computation in $\CC[V]$ is a generalized division algorithm using Groebner bases.  This requires an ordering on the monomials.  In one variable, the natural ordering is by degree: $1,z,z^2,z^3,...$.  For $n>1$, there are several natural generalizations, and we recall one called \emph{grevlex ordering}.

Let $\alpha,\beta$ be multi-indices in $N$ variables.  Writing $\prec$ for grevlex,  it is defined by setting $z^{\alpha}\prec z^{\beta}$ whenever:
\begin{itemize}
 \item $|\alpha|< |\beta|$; or
\item $|\alpha|=|\beta|$ and there exists $i\in\{1,...,n\}$ such that
\begin{itemize}
\item[] $\alpha_i>\beta_i$ and $\alpha_j=\beta_j$ for any positive integer $j<i$.  
\end{itemize}
\end{itemize}
For example, the first few monomials in $\CC[z_1,z_2]$  listed according to $\prec$ are
$$
1,z_1,z_2,z_1^2,z_1z_2,z_2^2,z_1^3,z_1^2z_2,z_1z_2^2,z_2^3,\ldots
$$
We will work exclusively with grevlex in what follows.  

We can now order terms of a polynomial $p(z)=\sum_{\alpha}a_{\alpha}z^{\alpha}$ unambiguously and define the \emph{leading term} 
$\lt(p) = a_{\beta}z^{\beta}$ to be the term for which $a_{\beta}\neq 0$ and for all $\alpha$ such that $a_{\alpha}\neq 0$, we have $z^{\alpha}\prec z^{\beta}$.

Consider the monomials as elements of $\CC[V]$, where $V$ is a curve.  We go through the monomials in increasing order (according to grevlex), throwing out linearly dependent monomials as they arise.  Let $H(s)=as+b$ be the Hilbert polynomial of $V$.  When $s$ is a sufficiently large positive integer, this says that our reduction process will keep $a$ monomials of degree $s$ and throw out the rest.

Let $\calB=\{z^{\alpha_j}\}_{j=1}^{\infty}$ denote our reduced collection of monomials, which forms a basis of $\CC[V]$.  Computation in terms of $\calB$ is done systematically using a Groebner basis of $\I(V)$, whose definition we now recall.  
First, given an ideal $I\subset\CC[z]$, let $\lt(I)=\{\lt(p): p\in I\}$, and denote the ideal generated from $\lt(I)$ by $\langle\lt(I)\rangle$.

\begin{definition}\rm 
A collection of polynomials $\{g_1,...,g_k\}$ is called a \emph{Groebner basis of $I$} if $\langle g_1,...,g_k\rangle=I$ and  $\langle\lt(g_1),...,\lt(g_k)\rangle=\langle\lt(I)\rangle$.
\end{definition}

More precisely, this defines a Groebner basis for the grevlex ordering. (Other orderings may  give different leading terms $\lt(g_i)$ for which the Groebner basis condition fails.)

We have the following result on computation in $\CC[V]$ (c.f., \cite{coxlittleoshea:ideals}, 5\S 3).

\def\CV{\CC[V]}
\begin{proposition}[Algebraic computation in $\CV$] \label{prop:comp}
\ \par  
\begin{enumerate}
\item A Groebner basis for $I=\I(V)$ always exists.
\item $\calB=\{z^{\alpha}: z^{\alpha}\not\in\langle\lt(I)\rangle\}$.
\item For any $p\in\CC[z]$ there is a unique $r\in\CC[z]$ such that
\begin{equation}\label{eqn:divalg}
p(z) = \sum_{j=1}^k q_j(z)g_j(z) \ + \ r(z)
\end{equation}
where $q_1,...,q_k\in\CC[z]$ and all terms of $r$ are in $\calB$. \qed  
\end{enumerate}
\end{proposition}


The representation (\ref{eqn:divalg}) is usually computed  using a generalized division algorithm,  in which  $\{g_1,...,g_k\}$ are the divisors, $\{q_1,...,q_k\}$ the quotients, and $r$ the remainder  (\cite{coxlittleoshea:ideals}, 2\S 3).  Although $r$ is unique, the quotients $q_i$ may not be.    

Note that for all  $z\in V$,  $p(z)=r(z)$; so $p(z)=r(z)$ in $\CC[V]$.   We call $r$ the \emph{normal form of $p$}.  For convenience, we will sometimes write $\rho_V(p)=r$.  The above proposition  says that $p\mapsto \rho_V(p)$ is a well-defined operation on polynomials, and corresponds to choosing the unique representative of the class of polynomials equal to $p$ on $V$ that can be expressed as a linear combination of elements of $\calB$.

\section{Computation on curves}  \label{sec:comp}

\subsection{Multiplication}  From Proposition \ref{prop:comp}, given $f,g\in\CC[z]$ it is easy to see that $f=g$ on $V$ if and only if $\rho_V(f)=\rho_V(g)$.  Rather than considering $\CC[V]$ as a factor ring, one can take the alternative view of $\CC[V]$ as the collection of normal forms.  This is the linear subspace of $\CC[z]$ spanned by $\calB$.  From this point of view, $\rho_V:\CC[z]\to \CC[V]$ is a linear map whose kernel is $\I(V)$.  Multiplication descends to a bilinear map $*:\CC[V]\times\CC[V]\to\CC[V]$ given by $$(p,q)\mapsto \rho_V(pq)=:p*q.$$  We will stick to this point of view in what follows: i.e., $\CC[V]$ is the space spanned by $\calB$, with a multiplication given by $*$.  The total degree can also be read off easily: if $p$ is in normal form, then  $\deg_V(p)=\deg(p)$ where the latter denotes the usual total degree in $\CC[z]$. 

\smallskip

Chebyshev constants that we will study in the next section are defined by fixing properties of leading homogeneous parts of polynomials.  Given $p=\sum_{|\alpha|\leq d}a_{\alpha}z^{\alpha}\in\CC[V]$, write $\hat p=\sum_{|\alpha|=d}a_{\alpha}z^{\alpha}$ for the leading homogeneous part of $p$.  

For polynomials $p$ and $q$ we also want to consider $\hat{p*q}$, the leading homogeneous part of the product.  Given a nonnegative integer $d$, write $\CC[V]_{=n}$ for the homogeneous polynomials of degree $n$ in $\CC[V]$.  Fixing $p$  with $\deg(p)=n$, we want  $q\mapsto\hat{p*q}$ to be a linear map $\CC[V]_{=n}\to\CC[V]_{=n+\deg p}$. But this is not always the case as cancellations may occur.
\begin{example}
 Let $V=\{z_2^2-z_1^2-z_1-1=0\}$.  Take $p=z_1+z_2$ and $q=z_1-z_2$.  Then $\hat{p*q} = z_1$ which is not of degree $2$.
\end{example}
To account for cancellation, we therefore define
$$
p\,\hat *\, q = \left\{\begin{array}{rl} \hat{p*q} & \hbox{ if } \deg p+\deg q =\deg(pq) \\
0 & \hbox{ otherwise}\end{array}\right. .
$$
Note that we let zero be an element of $\CC[V]_{=n}$ for each $n$ so that it becomes a vector space.  It is easy to see the following.
\begin{lemma}
Let $p$ be a fixed homogeneous polynomial.  Then $q\mapsto p\hsta q$ is a linear transformation $\CC[V]_{=n}\mapsto\CC[V]_{=n+\deg p}$ for any sufficently large positive integer $n$. \qed
\end{lemma}

When $V$ is a curve, there is a positive integer $d$, such that for sufficiently large degree $n$, $\CC[V]_{=n}$ is of  dimension $d$ (here $H(n)=dn+c$ ($d\in\ZZ_+,c\in\ZZ$) is the Hilbert polynomial), and $\CC[V]_{=n}$ has basis  $\{z^{\alpha}\not\in\langle\lt(I)\rangle:|\alpha|=n\}$.  In what follows, we will take $n$ to be sufficiently large that the homogeneous polynomials of a given degree form a space of dimension $d$.

The grevlex order gives an unambiguous representation of $\CC[V]_{=n}$ by $\CC^d$; if the basis of $\CC[V]_{=n}$ listed in (increasing) grevlex order is $\{z^{\alpha_1},...,z^{\alpha_d}\}$, then match $z^{\alpha_j}$ with the standard $j$-th coordinate in $\CC^d$.  Using this, we form vector and matrix representations of polynomials.  For $p\in\CC[V]_{=n}$ given by 
$
p(z)=a_{\alpha_1}z^{\alpha_1}+\cdots+a_{\alpha_d}z^{\alpha_d}, 
$
set 
\begin{equation}\label{eqn:d}  [p]:=(a_{\alpha_1},...,a_{\alpha_d})\in\CC^d.\end{equation}  
Similarly, denote by $[[p]]$ the representation of $p$ as a $d\times d$ matrix, i.e., as representing the linear map $q\mapsto{p\,\hsta\,q}$.   That is, $[[p]]$ is the matrix defined by the equation 
$$
 [[p]][q]\ = \ [p\hsta q].
$$

\subsection{Projective space} We will be interested in curves whose coordinate rings have additional nice properties for computation.  Computational properties of $\CC[V]$ are closely related to geometric properties of $V$.  

To study this relationship, we will consider $V$ as a curve in projective space $\CC\PP^N = \CC^N\cup H_{\infty}$, under the usual embedding $(z_1,...,z_N)\mapsto[1:z_1:\cdots:z_N]$ where the latter are homogeneous coordinates, i.e.,
$$
[z_0:z_1:\cdots:z_N]=[y_0:y_1:\cdots:y_N] \hbox{ if and only if } z_jy_k=z_ky_j \ \forall\, j,k=0,...,N.
$$

\emph{Dehomogenization (at $z_0$)} recovers affine coordinates via  
  $$
  [z_0:z_1:\cdots:z_N]=[1:z_1/z_0:\cdots:z_N/z_0]\mapsto (z_1/z_0,...,z_N/z_0), 
  $$ 
  at points of $\CC^N=\CC\PP^N\setminus H_{\infty}$.  Also, we will move relatively freely between standard coordinates $(z_1,...,z_n)$ and homogeneous coordinates $[z_0:z_1:\cdots:z_N]$.   Dehomogenization at $z_j$ (for other $j$) is defined similarly; this is useful to study points at infinity.

Recall that a \emph{projective variety} is a set of the form 
$$\{[z_0:z_1:\cdots:z_N]\in\CC\PP^n: q_1(z_0,...,z_N)=q_2(z_0,...,z_N)=\cdots=q_s(z_0,...,z_N)=0\}$$
where $\{q_j\}_{j=1}^s$ are homogeneous polynomials in $\CC[z_0,...,z_N]$.

\def\calW{\mathcal{W}}
  
  The \emph{homogenization (in  $z_0$) of $p\in\CC[z]$} is the unique homogeneous polynomial  $p^h\in\CC[z_0,z_1,...,z_N]=\CC[z_0,z]$ for which  $$\deg(p^h)=\deg(p)\ \hbox{ and  }\ p(z_1,z_2,...,z_N)=p^h(1,z_1,z_2,...,z_N).$$
Also, given an ideal $I\subset\CC[z_1,...,z_N]$, its 
\emph{homogenization  $I^h\subset\CC[z_0,z]$} is given by 
\begin{equation}\label{eqn:3.1a}
I^h = \langle p^h(z_0,z): p(z)\in I \rangle.
\end{equation}
The \emph{dehomogenization (at $z_0$)} of a homogeneous polynomial $h(z_0,z)$ is $h(1,z)$. 

\medskip

For our curve $V\subset\CC^N\subset\CC\PP^N$, define
$
V_{\PP}:=\bigcap\calV
$, 
where $$\calV\ = \ \{W\supset V:\ W \hbox{ is a projective variety in } \CC\PP^N\}.$$
We list some well-known properties of $V_{\PP}$ (c.f., \cite{coxlittleoshea:ideals}, 8\S 4).
\begin{proposition} \label{prop:3.3b}
Let $I=\I(V)$. 
\begin{enumerate}
\item  $V_{\PP}$ is the smallest projective variety containing $V$.  In particular, $V_{\PP}\setminus V$ is a finite subset of $H_{\infty}$.
\item   $V_{\PP}=\{[z_0:\cdots:z_N] : P(z_0,...,z_N)=0 \hbox{ for all } P\in I^h\}.    $  
\item \label{3.3(3)} If $G$ is a Groebner basis for $I$, then $G^h=\{g^h:g\in G\}$ is a Groebner basis of $I^h$ for the grevlex ordering on $\CC[z_0,z_1,...,z_N]$. 
 \qed
\end{enumerate} 
\end{proposition}

We call 
$V_{\PP}$ the \emph{projective closure of $V$ in $\CC\PP^N$}.

\begin{remark}\rm
 Part (\ref{3.3(3)}) of the proposition says that $\langle I^h\rangle = \langle G^h\rangle$.  It is essential that $G$ is a Groebner basis.  If not, then $I=\langle G\rangle$  only gives  $I^h\supset\langle G^h\rangle$.\end{remark}    


For the projective curve $V_{\PP}$, define    
\begin{equation}\label{eqn:Ih}
\I_{h}(V_{\PP}):= \{p\in\CC[z_0,...,z_N]: p(z_0,...,z_N)=0 \hbox{ whenever } [z_0:z_1:\cdots:z_N]\in V\}.
\end{equation}

\begin{remark}\rm
Write $p=\sum_{j=0}^{\deg p} p_j$ where for each $j$, $p_j$ is homogeneous of degree $j$. It is easy to see that $p\in\I_h(V_{\PP})$ if and only if $p_j\in\I_h(V_{\PP})$ for each $j$.  Hence we need only consider homogeneous polynomials in (\ref{eqn:Ih}).
\end{remark}

\begin{proposition} \label{prop:3.5}  Let $V\subset\CC^N$ be an algebraic curve with projective closure $V_{\PP}\subset\CC\PP^N$.  Then $\I_h(V_{\PP})$ is the homogenization of $\I(V)$. \qed
\end{proposition}

  Henceforth, we will conveniently write $V$ for $V_{\PP}$ and any dehomogenization of the latter, considering them  as the same object $V$ ``viewed projectively'' and ``viewed locally'' (e.g. write $\I_h(V)$ for $\I_h(V_{\PP})$).  
  
  \begin{lemma}
If $\langle G_1,...,G_s\rangle = \I_h(V)$ where $G_1,...,G_s\in\CC[z_0,...,z_N]$ are homogeneous polynomials,  then  $\langle g_1,...g_s\rangle =  \I(V)$ where $g_k(z_1,...,z_N) = G_k(1,z_1,...,z_N)$ for all $k=1,...,s$. \qed
  \end{lemma}
  
  We can study points of $V\cap\{z_j\neq 0\}$ by dehomogenizing the polynomials of $G$ at $z_j$ ($j\in\{1,...,N\}$).  This will be useful in the next section when studying $V$ near $H_{\infty}$.  If we start with a Groebner basis $G$ of $\I(V)$ in standard coordinates (for grevlex), then by Propositions  \ref{prop:3.3b} and \ref{prop:3.5},  homogenization of $G$ followed by dehomogenization at $z_j$ gives a Groebner basis for $V$ in local coordinates on $\{z_j\neq 0\}$. 
  
   Henceforth, we may refer to a Groebner basis $G$ as being associated to $V$, implicitly homogenizing and dehomogenizing the elements of $G$ as the context demands.  The associated ideal (e.g. $\I(V)$, $\I_h(V)$) will be clear from the context. 
  
  
\subsection{Algebraic and Geometric properties} \label{subsection:alg}
Let $V$ be a curve and let $I=\I(V)$ be its ideal in standard affine coordinates.  Computation in $\CC[V]$ simplifies when $\langle\lt(I)\rangle$ contains monomials of the form $z_k^{a_k}$ for each $k=2,...,N$.  The following proposition gives a condition under which this occurs.
\begin{proposition} \label{prop:2.5}
Suppose the curve $V$ satisfies the following condition on the coordinates of its points at infinity:
\begin{equation} \label{eqn:prop2.5}
[0:z_1:z_2:\cdots:z_N]\in V \ \Longrightarrow \ z_j\neq 0 \hbox{ for all } j=1,...,N.
\end{equation}
Then
\begin{enumerate}
\item \label{prop:2.5(1)} $z_1^a\not\in\lI$ for any positive integer $a$, and for each $k=2,...,N$ there is a positive integer $a_k$ such that $z_k^{a_k}\in \langle\lt(I)\rangle$.
\item \label{prop:2.5(2)} $[[z_1]]$ is the $d\times d$ identity matrix, $\Id$ (where $d$ is as in (\ref{eqn:d})).
\end{enumerate}
\end{proposition}

\begin{proof} 
First, suppose $z_1^a\in\langle\lt(I)\rangle$ for some $a$. Let $r(z)=\rho_V(z_1^a)$.  Then in $\CC[z_1,...,z_N]$, we have $\lt(r)\prec z_1^a$, which implies $\deg(r)<a$.  The equation $z_1^a=r(z)$ holds for any $z\in V$.  Projectively, this means 
$$
z_1^a = z_0^br^h(z) \hbox{ whenever } [z_0:z_1:\cdots:z_N]\in V;
$$
where $b=a-\deg(r)$ and $r^h\in\CC[z_0,z_1,...,z_N]$ is the homogenization of $r$.  Hence $z_0=0$ implies $z_1=0$.  This contradicts our hypothesis.  So $z_1^a\not\in\langle\lt(I)\rangle$ for any positive integer $a$.

Now suppose $z_2^a\not\in\langle\lt(I)\rangle$ for any positive integer $a$.  As before, let $d$ be the dimension of $\CC[V]_{=a}$ for $a$ sufficiently large.  If $a>d$, the number of monomials in $z_1,z_2$ of total degree $a$ is at least $a$; hence this collection is so large that at least one monomial $z_1^mz_2^n$ cannot be a monomial in $\CC[V]_{=a}$.  This is equivalent to $z_1^mz_2^n\in\langle\lt(I)\rangle$.  Consider the smallest such monomial with respect to grevlex, and let $r(z)=\rho_V(z_1^mz_2^n)$; then $\deg(r)<d$, and hence projectively,
$$
z_1^mz_2^n=z_0^br^h(z) \hbox{ for } [z_0:\cdots:z_N]\in V
$$
so that $z_0=0$ and $z_1\neq 0$ implies $z_2=0$, contradicting our hypothesis.  Hence for some $a\leq d$ we have $z_2^a\in\langle\lt(I)\rangle$.

The same argument as above can be repeated inductively for $z_3,...,z_N$.  In this case, if $z_1$ and $z_j$ can have arbitrarily large powers in $\CC[V]$ but the powers of $z_2,...,z_{j-1}$ remain bounded, then for a sufficiently large integer $a$ we can deduce the existence of a monomial $z^{\alpha}:=z_1^{a_1}z_2^{a_2}\cdots z_j^{a_j}\in\langle\lt(I)\rangle$ with $|\alpha|=a$ such that $\deg \rho_V(z^{\alpha})<a$, and by the same reasoning as above, $z_0=0$ and $z_2,...,z_{j-1}\neq 0$ must imply $z_j=0$, a contradiction.  This proves (\ref{prop:2.5(1)}).

To prove the second part, note that for each monomial $z^b=z_1^{b_1}z_2^{b_2}\cdots z_N^{b_N}\in\CC[V]$, we have $b_k\leq a_k$ whenever $k=2,...,N$.  Assume that the $a_k$'s are the minimum such integers, i.e., if $b_k>a_k$ for some $k=2,...,N$ then $z^b\in\langle\lt(I)\rangle$; and set $a_1=\infty$ for convenience.  Thus the condition $b_k\leq a_k$ for all $k$ characterizes the monomials $z^b\in\CC[V]$. 
 
It follows that on monomials the map $z^b\mapsto z_1z^b$
is a bijection from $\CC[V]_{|b|}$ to $\CC[V]_{|b|+1}$, since in $\CC[z]$, the fact that  $b_k>a_k$ or $b_k\leq a_k$  remains the same for the image under this map.   Since the grevlex ordering is also unaffected, the representation of $[[z_1]]$ must be the identity, giving (\ref{prop:2.5(2)}).
\end{proof}

The first part of the above proposition has a partial converse.

\begin{proposition}
Suppose $\lI$ satisfies Proposition \ref{prop:2.5}(\ref{prop:2.5(1)}).  Then $z_1\neq 0$, i.e.,
$$
V\cap \{[0:0:z_2:\cdots:z_N] : z_j\in\CC\} = \emptyset.
$$
\end{proposition}

\begin{proof}
Suppose $z_2^{a_2}\in\langle\lt(I)\rangle$.  Then it must be generated by a leading term of a Groebner basis polynomial; in fact, if $a_2$ is the minimum such positive integer then there is a Groebner basis polynomial of the form
$$
g(z) = z_2^{a_2} + z_1q_1(z) + q_0(z)
$$
where $\deg q_0,\deg q_1<a_2$.  The polynomial $z_1q_1$ consists of the rest of the terms in the leading homogeneous part of $g$ (which according to grevlex comprise powers of $z_1$ and $z_2$ only), and $q_0$ is the lower degree terms.  Projectively, we have for all $[z_0:z_1:\cdots:z_N]\in V$ that
$$
0=g^h(z_0,...,z_N) = z_2^{a_2} + z_1z_0^{c_1}q_1^h(z_0,...,z_N) + z_0^{c_2}q_0^h(z_0,...,z_N)
$$
where we write $c_1=a_2-1-\deg q_1$, $c_2=a_2-\deg q_2$.  When $z_0=z_1=0$ the above equation reduces to $z_2^{a_2}=0$, hence $z_2=0$.  So $$V\cap \{[0:0:z_2:\cdots:z_N]\} = V\cap \{[0:0:0:z_3\cdots:z_N]\}.$$

Using our hypothesis again, we can find by the same process a Groebner basis polynomial whose homogenization has the form $z_3^{a_3}+ z_2q_2 + z_1q_1 + z_0q_0$ from which we can show that if $z_0=z_1=z_2=0$ then $z_3=0$ for all points on $V$.

This forms the basis of an inductive argument that results in the statement that 
$$z\in V\cap\{[0:0:z_2:\cdots:z_N]\} \ \Longrightarrow \ z=[0:\cdots:0].$$ Since no such $z$ in $\CC\PP^N$ exists,  $V\cap\{[0:0:z_2:\cdots:z_N]\}=\emptyset$.
\end{proof}

\begin{proposition} \label{prop:3.7}
Suppose (\ref{eqn:prop2.5}) holds.  Then for $j=2,...,N$, the following are equivalent:
\begin{enumerate}
\item $\lambda$ is an eigenvalue of $[[z_j]]$.
\item There is $[0:1:\lambda_2\cdots:\lambda_N]\in V\cap H_{\infty}$ such that $\lambda=\lambda_j$.
\end{enumerate}
Hence the matrices $[[z_j]]$ are nonsingular for all $j$.
\end{proposition}

\begin{proof}
\underline{(1) $\Rightarrow$ (2)}   Suppose $\lambda$ is an eigenvalue of $[[z_j]]$.  Then there exists a nonzero homogeneous polynomial $v\in\CC[V]$ with the property that $[[z_j]][v]=\lambda[v]$.  Since $[[z_1]]$ is the identity matrix $\Id$ (Proposition \ref{prop:2.5}(\ref{prop:2.5(2)})), we have 
$\left([[z_j]]-\lambda[[z_1]]\right)[v] = 0$, and so 
$$(z_j-\lambda z_1)v(z_1,...,z_N) = z_0r(z_0,z_1,...,z_N) \ \hbox{ for } [z_0:\cdots:z_N]\in V
$$
for some homogeneous polynomial $r$.
Thus $(\lambda_j-\lambda)v(1,\lambda_2...,\lambda_N)=0$ whenever $[0:1:\lambda_2:\cdots:\lambda_N]\in V$.  If $\lambda_j=\lambda$ at one of these points, we are done.  

Otherwise, we aim to derive a contradiction.  Suppose $[0:1:\lambda_2:\cdots:\lambda_N]\in V$ implies $\lambda_j\neq\lambda$.  Then $v(1,\lambda_2,...,\lambda_N)=0$ always holds.  By homogeneity, this means $v(z_1,z_2,...,z_N)=0$ whenever $[0:z_1:\cdots:z_N]\in V$, and so
$$
v(z_1,...,z_N)=z_0r(z_0,z_1...,z_N)\ \hbox{ for all } [z_0:z_1:\cdots:z_N]\in V
$$
for some homogeneous polynomial $r(z_0,...,z_N)$.  Viewed affinely in $\CC[z_1,...,z_N]$, this  implies that $\lt(v(z_1,...,z_N)-r(1,z_1,...,z_N))=\lt(v(z_1,...,z_N))$.  But we have
$$
v(z_1,...,z_N)-r(1,z_1,...,z_N) = \sum_{g_i\in G} q_ig_i  \ \hbox{ in }\CC[z_1,...,z_N],
$$
where we sum on the right-hand side over a Groebner basis $G$ for $I=\I(V)$.  Equating coefficients in this equation yields $\lt(v)\in\langle\lt(G)\rangle=\lI$, but this contradicts the fact that $v$ is a normal form.  Therefore, the statement that $\lambda_j\neq\lambda$ whenever $[0:1:\lambda_2:\cdots:\lambda_N]\in V$ is false.

\underline{(2) $\Rightarrow$ (1)} Suppose $[0:1:\lambda_2:\cdots:\lambda_N]\in V$.  Let $P$ be the characteristic polynomial of $A=[[z_j]]$, and $P^h$ its homogenization (in one more variable), such that $P(\lambda)=P^h(1,\lambda)$.  By the Cayley-Hamilton theorem of linear algebra a matrix satisfies its characteristic equation, so   $0=P(\Id,[[z_j]])=P^h([[z_1]],[[z_j]])$.  Translated back to computation on $V$, this says that 
$$
P^h(z_1,z_j) = z_0r(z_0,z_1,...,z_N) \ \hbox{ whenever } [z_0:z_1:\cdots:z_N]\in V,
$$
for some polynomial $r$.  Plugging in $[0:1:\lambda_2:\cdots:\lambda_N]$ to the above equation gives $0=P^h(1,\lambda_j)=P(\lambda_j)$.  Hence $\lambda_j$ is an eigenvalue of $[[z_j]]$.

Note that since (\ref{eqn:prop2.5}) holds, clearly $[[z_j]]$ is nonsingular since all of its eigenvalues are nonzero. 
\end{proof}

Recall that a curve $V$ intersects a hyperplane $H$ \emph{transversally} at a point $p$ if any tangent line to $V$ at $p$ does not lie in $H$.  The following is a straightforward consequence of Proposition \ref{prop:2.1}. 
\begin{lemma}
Suppose $V\subset\CC\PP^N$ is a curve and $\langle g_1,...,g_s\rangle = \I_{h}(V)$, where $g_k$ is a homogeneous polynomial for each $k=1,...,s$.   Then 
\begin{enumerate}
\item For any nonsingular point $a$, the matrix $J_a(g_1,...,g_s)$ has rank $N-1$ in any local coordinate.\footnote{ e.g. dehomogenize $g_1,...,g_s$ at $z_j$ if $a\in\{[z_0:\cdots:z_N]:z_j\neq 0\}$.} \addtocounter{footnote}{-1}

\item  If $V$  intersects the hyperplane $$H=\{[z_0:\cdots:z_N]:h(z_0,...,z_N)=A_0z_0+\cdots+A_Nz_N=0\}$$ transversally at $p$, then in any local coordinate at $p$, $J_p(g_1,...,g_s,h)$ has rank $N$. \qed
\end{enumerate}
\end{lemma}

We have the following. 

\begin{proposition} \label{prop:2.8}
Suppose $V\cap H_{\infty}$ is a set of $d$ points  $$\{\lambda_i=[0:1:\lambda_{i,2}:\cdots:\lambda_{i,N}]\}_{i=1}^d,$$ each of which is a nonsingular point of $V$ that intersects $H_{\infty}$ transversally.   Suppose for each $i=1,...,d$ and $j=2,...,N$, $\lambda_{i,j}\neq 0$.  Further, suppose for each $j\in\{2,...,N\}$, no two distinct points of $V\cap H_{\infty}$ have the same $j$-th coordinate (i.e. $i\neq i'$ implies $\lambda_{i,j}\neq\lambda_{i',j}$). 
  
  Then $\dim\CC[V]_{=n}=d$ and   
   $[[z_j]]$ is a $d\times d$ matrix with eigenvalues $\{\lambda_{i,j}\}_{i=1}^d$.  Hence the eigenvalues are  all of multiplicity one.  
   \end{proposition}

\begin{proof}
By Proposition \ref{prop:3.7}, we have exactly $d$ distinct eigenvalues of $[[z_j]]$, which are given by the $\lambda_{i,j}$, $i=1,...,d$.  Hence $\dim\CC[V]_{=n}\geq d$.  

We need to verify that these eigenvalues all have multiplicity one.  The argument is the same for each, and proceeds by contradiction.  Given $i$, suppose that $\lambda_{i,j}$ has multiplicity $l>1$.  From the Cayley-Hamilton theorem, $P(\Id,[[z_j]])=0$, where $P$ is the homogenization in the first variable of the characteristic polynomial of $[[z_j]]$. Then $\lambda_{i,j}$ being of multiplicity $l$ says that   $P(1,\lambda)=(\lambda-\lambda_j)^lp(1,\lambda)$.  This translates to the statement that 
\begin{equation}\label{eqn:2.2aa}
(z_j-\lambda_{i,j}z_1)^lp(z_1,z_j) + z_0r(z_0,z_1,...,z_N) \ = \ 0 \ \hbox{for all} \ [z_0:z_1:\cdots:z_N]\in V 
\end{equation}
for some homogeneous polynomial $r$ with $\deg(r)=\deg(P)-1$.

Let $Q(z_0,z_1,z_2,...,z_N):=(z_j-\lambda_{i,j}z_1)^lp(z_1,z_j) + z_0r(z_0,z_1,z_2,...,z_N)$.  Taking partial derivatives in $z_0,z_2,...,z_N$ and evaluating at the point $\lambda_i$ (i.e., set   $z_0=0$, $z_1=1$ and $z_j=\lambda_{i,j}$ for $j=2,...N$), we obtain
\begin{equation}\label{eqn:3.3a}
\frac{\partial Q}{\partial z_j}(\lambda_i)=0 \hbox{ for } j=2,...,N; \ \frac{\partial Q}{\partial z_0}(\lambda_i)=r(\lambda_i).
\end{equation}

For the rest of the proof, we assume that $r(\lambda_i)\neq 0$; this will be justified in the remark that follows.  

We have $\I_{h}(V)=\langle g_1,...,g_s\rangle$ for some homogeneous polynomials $g_1,...,g_s$.  By (\ref{eqn:2.2aa}), $Q\in\I_{h}(V)$, so $\I_{h}(V)=\langle g_1,...,g_s,Q\rangle$.  By the first part of the previous lemma, $J_{\lambda_i}(g_1,...,g_s,Q)$ has rank $N-1$ in local coordinates.

On the other hand, by (\ref{eqn:3.3a}), $J_{\lambda_i}(g_1,...,g_s,Q)=J_{\lambda_i}(g_1,...,g_s,r(\lambda_i)z_0)$.  But by the second part of the previous lemma, this has rank $N$ since $H_{\infty}$ intersects $V$ transversally.  This contradicts the previous paragraph.

Hence $\lambda_{i,j}$ has multiplicity one as an eigenvalue of $[[z_j]]$.

Since each eigenvalue of $[[z_j]]$ has multiplicity 1, the characteristic polynomial of $[[z_j]]$ has degree $d$.  So  $[[z_j]]$ is a $d\times d$ matrix and  $\dim\CC[V]_{=n}=d$. 
\end{proof}

\begin{remark} \rm
The assumption that $r(\lambda_i)\neq 0$ may be justified by translating $V$ in the $z_1$ direction.  This amounts to replacing $z_1$ by $z_1+\alpha z_0$ ($\alpha\in\CC$) in homogeneous coordinates in equation (\ref{eqn:2.2aa}).  Such a translation $V_{\alpha}$ will affect the terms of lower degree in affine coordinates (represented by the polynomial $r$); we simply arrange the value of $\alpha$ so that $r(\lambda_i)\neq 0$.  Note that if $G=\{g_1,...,g_s\}$ is a Groebner basis for $\I(V)$, then $\tilde G=\{\tilde g_1,...,\tilde g_s\}$ is a Groebner basis for $\I(\tilde V_{\alpha})$, where $\tilde g_k(z_1,z_2,...,z_N)=g_k(z_1+\alpha,z_2,...,z_N)$.   Since $\lt(G)=\lt(\tilde G)$,
 the derived matrices $[[z_j]]$ are exactly the same for $V_{\alpha}$ as for $V$.
\end{remark}

\begin{remark}\rm
Proposition \ref{prop:2.8} says that the intersection number of $V$ with a hyperplane is the same as the leading coefficient of its (linear) Hilbert polynomial.  This can be extracted as a special case of a general formula in algebraic geometry (cf. Theorem 7.7 of \cite{hartshorne:algebraic}).  This number $d$ is called the \emph{degree of $V$}, denoted $\deg(V)$.
\end{remark}

\begin{lemma} \label{lem:2.9a}
Suppose $V$ satisfies the hypotheses of Proposition \ref{prop:2.8}.  Let $V_j\subset\CC\PP^2$ be the curve given by projecting $V$ to the coordinates $[z_0:z_1:z_j]$.  Then there is a homogeneous polynomial $P_j\in\CC[z_0,z_1,z_j]$ such that $\deg P_j=d$, $\I(V_j)=\langle P_j\rangle$, and $\lambda\mapsto P_j(0,1,\lambda)$ is the characteristic polynomial of $[[z_j]]$.
\end{lemma}

\begin{proof}
Since $V_j$ is a curve in $\CC\PP^2$, we have $V_j=\{P_j=0\}$ and $\I(V_j)=\langle P_j\rangle$  for some polynomial $P_j\in\CC[z_0,z_1,z_j]$.  That $\deg(P_j)=d$ follows from our hypotheses on $V$: $V_j\cap\{z_0=0\}$ contains precisely $d$ points, and projection does not increase the degree of a curve.  So $\deg(P_j)\leq d$.  It is easy to see equality by applying the one-variable Factor Theorem to $P_j$ (setting $z_0=0$ and $z_1=1$).

Let $Q_j(z_1,z_j)$ be the homogenization in two variables of the characteristic polynomial of $[[z_j]]$, i.e., $\lambda\mapsto Q_j(1,\lambda)$ is the characteristic polynomial of $[[z_j]]$.  Factoring out $\lambda-\lambda_j$ from the characteristic polynomial and translating the characteristic equation to computation on $V$ (using the Cayley-Hamilton theorem as in the proof of Proposition \ref{prop:2.8}), we have
$$
Q_j(z_1,z_j) = (z_j-\lambda_j z_1)v_j(z_1,z_j) = 0 \ \hbox{ on }  V\cap H_{\infty} 
$$
for some homogeneous polynomial $v_j(z_1,z_j)$ of degree $d-1$.  Since the above equation is  independent of all coordinates other than $z_0,z_1,z_j$ we can consider it on the projection $V_j\subset\CC\PP^2$.  Then $(z_j-\lambda_j z_1)v_j(z_1,z_j) = 0$ on $V_j\cap\{z_0=0\}$, which  says that 
\begin{equation}\label{eqn:2.4aa}
Q_j(z_1,z_j)=(z_j-\lambda_j z_1)v_j(z_1,z_j)=z_0r_j(z_0,z_1,z_j) \ \hbox{ if } [z_0:z_1:z_j]\in V_j, 
\end{equation}
where $r_j$ is a homogeneous polynomial of degree $d-1$.  Now (\ref{eqn:2.4aa}) says that  $Q_j(z_1,z_j)-z_0r_j(z_0,z_1,z_j)\in\langle P_j\rangle$.  Since $\deg(Q_j)=\deg(P_j)=d$, $Q_j(z_1,z_j)$ must be a constant multiple of $P_j(0,z_1,z_j)$.  Renormalizing $P_j$ yields the result.
\end{proof}

Given $j\in\{2,...,N\}$, suppose $\lambda_j$ is an eigenvalue of $[[z_j]]$.  Then there is an associated eigenvector, which we translate into a homogeneous polynomial as follows.  Let $n$ be the smallest positive integer such that $\CC[V]_{=n}$ has dimension $d$.  Set $v_{\lambda_j}(z) \ := \ \sum_{k=1}^d a_kz^{\alpha_k}$ 
where $(a_1,...,a_d)$ is an eigenvector of $\lambda_j$ and $\{z^{\alpha_1},...,z^{\alpha_d}\}$ is the basis for $\CC[V]_{=n}$.  (So $[[z_j]][v_{\lambda_j}]=\lambda_j[v_{\lambda_j}]$.)  


\begin{lemma} \label{lem:2.9}
Suppose $V$ satisfies the hypotheses of Proposition \ref{prop:2.8}, and $v_{\lambda_j}$ is a polynomial, as defined above, associated to some eigenvalue $\lambda_j$ of $[[z_j]]$.  For $w=(1,w_2,...,w_N)$, suppose $[0:w]=[0:1:w_2:\cdots:w_N]\in V\cap H_{\infty}$.  

Then
$v_{\lambda_j}(w) = 0$ iff $w_j\neq\lambda_j$.
\end{lemma}

\begin{proof}
From $[[z_j]][v_{\lambda_j}]=\lambda_j[v_{\lambda_j}]$ we obtain for $z=[1:z_1:\cdots:z_N]\in V$ that
$z_jv_{\lambda_j}(z) = \lambda_j z_1v_{\lambda_j}(z)+r(z)$ for some polynomial $r$ with  $\deg(r)\leq\deg(v_{\lambda_j})$, and hence
$$
(z_j-\lambda_j z_1)v_{\lambda_j}(z) = z_0^{a}r^h(z_0,z_1,...,z_N) \ \hbox{whenever } [z_0:z_1:\cdots:z_N]\in V
$$
for some positive integer $a$.  Plugging in $[0:w]\in V\cap H_{\infty}$, we obtain
$(w_j-\lambda_j)v_{\lambda_j}(w)=0$.  Hence $w_j\neq\lambda_j$ implies $v_{\lambda_j}(w)=0$. 

We now show that $w_j=\lambda_j$ implies $v_{\lambda_j}(w)\neq 0$.  As before, let $V_j$ be the algebraic curve in $\CC\PP^2$ given by projecting $V$ to the coordinates $[z_0:z_1:z_j]$,  with $V_j=\{P_j=0\}$, $\deg(P_j)=d$.  Let $v_j$ and $r_j$ be as in (\ref{eqn:2.4aa}).  Since $v_j$ is formed by factoring out $(z_j-\lambda_j z_1)$ where $\lambda_j$ is a simple eigenvalue of $[[z_j]]$, we have $v_j(1,\lambda_j)\neq 0$.  This will give us what we want after transferring our calculations back to $V$.  To this end, set 
$$v(z)=v(z_1,...,z_n):=\rho_V(v_j(z_1,z_j)), \ \hbox{ and } \  r(z)=r(z_1,...,z_n):=\rho_V(r_j(1,z_1,z_j))$$ in $\CC[z_1,...,z_N]$.    
Equation (\ref{eqn:2.4aa}) then translates to 
\begin{equation}\label{eqn:2.2}
z_jv(z)\ = \ \lambda_jz_1v(z) + r(z), \hbox{ for all } [1:z_1:\cdots:z_N]\in V.
\end{equation}
We verify that $\deg v=\deg v_j$.  Clearly $\deg v\leq\deg v_j$\footnote{Since grevlex is a graded order, the computation of $v$ from $v_j$ does not increase degree. For the same reason, $\deg(r_j)\geq\deg(r)$.} so that  $$z_0^cv^h(z_0,z_1,...,z_N)-v(z_1,z_j)=0\ \hbox{ for all }  [z_0:z_1:\cdots:z_N]\in V,$$ where $c=\deg v_j-\deg v$ and $v^h$ is the homogenization of $v$ in the variable $z_0$.  But if $c>0$, then evaluating at a point  $[0:1:\lambda_2:\cdots:\lambda_N]\in V$ would give $v_j(1,\lambda_j)=0$, a contradiction.  

Hence $\deg v=\deg v_j>\deg r_j\geq\deg r$ and \begin{equation}\label{eqn:314a}0\neq v_j(1,\lambda_j) = v^h(0,1,\lambda_2,...,\lambda_N) = \hat v(1,\lambda_2,...,\lambda_N)\end{equation} where $v^h$ is as above  and $\hat v$ is the leading homogeneous part of $v$ in the variables $z_1,z_2,...,z_n$. 

We now rewrite equation (\ref{eqn:2.2}), grouping the lower degree terms of $v$ with $r$ to obtain 
$$
z_j\hat v(z) = \lambda_jz_1\hat v(z) + \tilde r(z_1,...,z_N) \hbox{ for all } [1:z_1:\cdots:z_N]\in V,
$$
with $\deg\tilde r\leq\deg\hat v$.  The above equation says that $[\hat v]$ is an eigenvector  of $[[z_j]]$, therefore using the uniqueness of eigenvectors of multiplicity one up to scalar multiples, $\hat v = cz_1^av_{\lambda_j}$ for some nonzero constant $c$ and non-negative integer $a$.  Plugging in $[0:w]=[0:1:\lambda_2 :\cdots:\lambda_N]\in V$ and using (\ref{eqn:314a}), we have
$$
0\neq v_j(1,\lambda_j) = \hat v(w) = cv_{\lambda_j}(w),
$$
so $v_{\lambda_j}(w)\neq 0$, as required.
\end{proof}

To get a unique polynomial associated to $\lambda_j$ we can choose a convenient normalization.  We will normalize as follows: put \begin{equation}\label{eqn:2.6a}v_{\lambda_j}(\lambda)=1\end{equation} where  $\lambda=[0:1:\lambda_2:\cdots:\lambda_N]$ is the unique point of $V\cap H_{\infty}$ whose $j$-th coordinate is $\lambda_j$.  Let us call this normalized polynomial $v_{\lambda_j}$ the \emph{eigenvector polynomial associated to $\lambda_j$.}

\smallskip

When a curve $V$ is reducible, an   
 eigenvector polynomial for one of its components is related to one for the entire curve.

\begin{proposition}\label{prop:2.10}
Suppose $V=V_1\cup V_2$ where $V_1,V_2$ are algebraic curves of degree $d_1,d_2$, with $d_1+d_2=d$.  Let $\lambda=[0:1:\lambda_1:\cdots:\lambda_N]\in V_1\cap H_{\infty}$.  For $j\in\{1,...,N\}$, let $w_j\in\CC[V_1]$, $v_j\in\CC[V]$ be the eigenvector polynomials for $\lambda_j$ on $V_1$ and $V$ respectively.  Then there is a homogeneous polynomial $\varphi\in\CC[V]$ and nonnegative integer $a$ such that
\begin{equation}\label{eqn:2.3}
z_1^{a}v_j(z) = w_j(z)\hsta \varphi(z).
\end{equation}
holds in $\CC[V]$.
\end{proposition}

\begin{remark}\rm
By construction, monomials in $\CC[V_1]$ are also in $\CC[V]$.  Hence on the right-hand side of (\ref{eqn:2.3}),  $w_j$ also makes sense as a polynomial in $\CC[V]$.  
\end{remark}

\begin{proof}
Let $\pi:\CC\PP^N\to\CC\PP^2$ be the projection to the coordinates $[z_0:z_1:z_j]$.  Then $\pi(V)=\pi(V_1)\cup\pi(V_2)$ with $\pi(V_2) = \{[z_0:z_1:z_j]:P(z_0,z_1,z_j)=0\}$ for some homogeneous polynomial $P$.  Let $\varphi_0(z):=\rho_V(P(1,z_1,z_j))\in\CC[V]$, with leading  homogeneous part $\hat\varphi_0(z)$.  

We show that $w_j\hsta\hat\varphi_0$ satisfies the eigenvector property in $\CC[V]$ for $\lambda_j$.  To show this, let  $z\in V$.  If $z\in V_1$, then we use computational properties of $w_j$ in $\CC[V_1]$ (in particular, the equivalent of (\ref{eqn:2.2}) for $w_j$ on $V_1$):
\begin{eqnarray*}
z_j(w_j\hsta\hat\varphi_0)(z) &=& z_jw_j(z)\varphi_0(z) + r_1(z) \\
&=& (\lambda_jz_1w_j(z) + r_2(z))\varphi_0(z) + r_1(z) \\
&=& \lambda_jz_1w_j(z)\hsta\hat\varphi_0(z) + r(z) = \lambda_jz_1v(z) + r(z)
\end{eqnarray*}
where $\deg(r_1)\leq\deg v$, $\deg(r_2)\leq\deg w_j$ and $\deg r\leq\deg v$.

 On the other hand, if $z\in V_2$, then
\begin{eqnarray*}
z_j(w_j\hsta\hat\varphi_0)(z) 
&=& z_jw_j(z)\varphi_0(z) \ +  \  r_3(z) \\
&=&  \lambda_j z_1w_j(z)\varphi_0(z)  \ + \  r_3(z) 
\end{eqnarray*}
where $\deg r_3\leq\deg (w_j\hsta\hat\varphi_0)$.  Note that the first term on the right-hand side of each line is zero since $P(z_0,z_1,z_j)=0$ on $\pi(V_2)$ implies $\varphi_0(z)=0$ on $V_2$.  We have simply used this fact to cast the expression into the form we want (i.e. replacing $z_j$ with $\lambda_jz_1$).  

In summary, this shows that $[w_j\hsta\hat\varphi_0]$ is an eigenvector for $[[z_j]]$ for $\lambda_j$ (representing computation in $\CC[V]$).  By uniqueness of eigenvectors up to scalar multiples (since $\lambda_j$ has multiplicity one), $C[w_j\hsta\hat\varphi_0]=[v_j]$ where $v_j$ is the eigenvector polynomial in $\CC[V]$ for $\lambda_j$ and $C\in\CC$ is a constant.  So  
$
z_1^av_j(z) =  w_j\hsta C\hat\varphi_0.
$  
for some non-negative integer $a$.  Setting $\varphi:=C\hat\varphi_0$ gives the result.
\end{proof}

\subsection{An illustration} \label{sec:illustration}

We illustrate the concepts of Section \ref{subsection:alg} with a concrete example.  Consider the curve $V$ in $\CC^3$ given by $$z_2^2+z_3^2-z_1^2-1=z_3^2+z_2z_3-2z_2^2+z_1z_3-z_1z_2+1=0.$$  A Groebner basis for $\I(V)$ (for grevlex) is given by\footnote{All calculations were done in practice by a computer algebra system.}  \addtocounter{footnote}{-2}
\begin{eqnarray*}
G&=&\bigl\{z_2z_3+z_1z_3-3z_2^2-z_1z_2+z_1^2+2, z_3^2+z_2^2-z_1^2-1, \\ & \ & \hskip1cm 10z_2^3-2z_1z_2^2-6z_1^2z_2+z_1^2z_3+z_1^3-7z_2-2z_3+3z_1\bigr\}.
\end{eqnarray*}
This gives $\langle\lt(G)\rangle = \langle z_2z_3, z_3^2, 10z_2^3 \rangle$.  From this we obtain that the monomial basis of $\CC[V]_{=n}$ (for $n\geq 2$) is $\{z_1^n,z_1^{n-1}z_2,z_1^{n-1}z_3,z_1^{n-2}z_2^2\}$.  Clearly $[[z_1]]$ is the $3\times 3$ identity matrix with respect to this basis.  More calculations yield 
$$
[[z_2]] = \begin{bmatrix}
  0&0&-1&  -1/10\\
1&0&1&{6/10}\\         
  0&0&-1&{-1/10}\\
0&1&3&{2/10}
          \end{bmatrix}, \quad
[[z_3]] = 
\begin{bmatrix}
0&-1&1&7/10 \\
0&1&0&-2/10\\
1&-1&0&7/10\\
0&3&-1&-14/10
\end{bmatrix}.          
$$
Observe that $[0:1:\frac{1}{\sqrt{2}}:\frac{1}{\sqrt{2}}]$ is a point in the projective closure of $V$ (using homogeneous coordinates $[z_0:z_1:z_2:z_3]$), and that $\frac{1}{\sqrt{2}}$ is an eigenvalue of both $[[z_2]]$ and $[[z_3]]$.

\section{Chebyshev constants} \label{sec:cheby}

We define notions of Chebyshev constant associated to a compact subset of an algebraic curve.

  Let $V$ be an algebraic curve of degree $d$ whose points at infinity satisfy the hypotheses of Proposition \ref{prop:2.8}.
    Recall that such an algebraic curve has  the following properties:
\begin{enumerate}[(i)]
\item We have $(V\cap H_{\infty})\subset\{z_1\neq 0\}$ in homogeneous coordinates $[z_0:z_1:\cdots:z_N]$.
\item The intersection points $V\cap H_{\infty}$ are nonsingular on $V$ and all intersections are transverse.
\item If $\lambda_1=[0:1:\lambda_{12}:\cdots\lambda_{1N}]$ and $\lambda_2=[0:1:\lambda_{22}:\cdots:\lambda_{2N}]$ are points of $V\cap H_{\infty}$, then $\lambda_{1j}\neq\lambda_{2j}$ for all $j=2,...,N$.
\end{enumerate}  
  
  \medskip
  
Let $\lambda=[0:1:\lambda_2:\cdots:\lambda_N]\in V\cap H_{\infty}$ be a point at infinity.  We have the following.

  
  \begin{lemma}\label{lem:3.1}  There is a unique polynomial $\v_{\lambda}\in\CC[V]$ of minimal degree with the following properties:
  \begin{enumerate}
  \item $\v_{\lambda}(\lambda)=1$. 
  \item $\v_{\lambda}(w)=0$ if $w\in (V\cap H_{\infty})\setminus\{\lambda\}$.
  \item For any polynomial $p\in\CC[V]$, \begin{equation}\label{eqn:3.1}(p*\v_{\lambda})(z) \ = \ z_1^{\deg(p)}\hat p(\lambda)\v_{\lambda}(z)   +   r(z)\end{equation} for some polynomial $r\in\CC[V]$ with $\deg(r)<\deg(p)+\deg(\v_{\lambda})$.  
  \item If (1)--(3) hold with $\v_{\lambda}$ replaced by some polynomial $w$, then $w=z_1^a\v_{\lambda}$ for some non-negative integer $a$.
  \end{enumerate}
   \end{lemma}
   
  \begin{proof}
  Set 
$$
\tilde\v_{\lambda}(z) \ := \ (v_2\hsta v_3\cdots\hsta v_N)(z),
$$
where  $v_j(z)$ is the eigenvector polynomial in $\CC[V]$ corresponding to the eigenvalue $\lambda_j$ of $[[z_j]]$, normalized so that $v_{j}(\lambda)=1$.
  
Then by Lemma \ref{lem:2.9} and the normalization equation (\ref{eqn:2.6a}), $\tilde\v_{\lambda}(z)$ satisfies the first two properties.  By linearity, it suffices to verify the third property when $p$ is a monomial.  This is a calculation that uses the fact that $\tilde\v_{\lambda}$ is formed from eigenvector polynomials.  Explicitly, given $p(z)=z_1^{\alpha_1}z_2^{\alpha_2}\cdots z_N^{\alpha_N}$ (so $\deg(p)=|\alpha|$), then for $z\in V$, we have by repeated application of (\ref{eqn:2.2}) that 
\begin{eqnarray*}
p(z)\tilde\v_{\lambda}(z) &=& z_1^{\alpha_1}(\prod_{j=2}^N z_j^{\alpha_j})(\prod_{j=2}^N v_{j}(z))  
\ = \  z_1^{\alpha_1}\prod_{j=2}^N z_j^{\alpha_j}v_{j}(z) \\
&= & z_1^{\alpha_1}\prod_{j=2}^N \Bigl((\lambda_jz_1)^{\alpha_j}v_{j}(z) + r_j(z)\Bigr)  \quad (\hbox{with } \deg(r_j)<\deg(v_j)+\alpha_j)\\
&=& \Bigl(z_1^{\alpha_1+\cdots+\alpha_N}\prod_{j=2}^N \lambda_j^{\alpha_j}\prod_{j=2}^N v_{j}(z)\Bigr)\ + \  r(z) 
\ = \ z_1^{|\alpha|}\hat p(\lambda)\tilde\v_{\lambda} \ + \ r(z), 
\end{eqnarray*}
with $\deg(r)<|\alpha|+\deg(\tilde\v_{\lambda})$.  This proves property (3).

Suppose $w\in\CC[V]$ is a homogeneous polynomial satisfying the first three properties; then
$$
z_1^{\deg(w)}\tilde\v_{\lambda}(z) + r(z) = w(z)*\tilde\v_{\lambda}(z)=z_1^{\deg(\tilde\v_{\lambda})}w(z) +\tilde r(z)
$$
where $\deg(r),\deg(\tilde r)<\deg(w*\tilde\v_{\lambda})$.  
Since the first and last polynomials are identical, equating coefficients gives 
\begin{equation}\label{eqn:4.2ab}
z_1^{\deg(w)}\tilde\v_{\lambda}(z) =z_1^{\deg(\tilde\v_{\lambda})}w(z).
\end{equation} \def\calW{\mathcal{W}}

The collection $\calW$ of all homogeneous polynomials $w\in\CC[V]$ satisfying the first three properties is thus a nonempty subset of $$\{p\in\CC[z_1,...,z_N]: p(z)=z_1^a\tilde\v_{\lambda} \hbox{ for some } a\in\ZZ\},$$ which is well-ordered by (total) degree.  By the well-ordering principle, we can take $\v_{\lambda}\in\calW$ to be the element with minimal degree.  It is unique and satisfies property (4).
\end{proof}

\begin{definition}\rm
We will call $\v_{\lambda}$  the \emph{directional polynomial} for $\lambda$.
\end{definition}

\begin{example}\rm
Let $V$ be the algebraic curve in $\CC^3$ given by 
$$
z_2^2+z_3^2-z_1^2-1 = z_3^2+z_2z_3-2z_2^2+z_1z_3-z_1z_2+1 = 0,
$$
which was considered in Section \ref{sec:illustration} above.  One can verify that $\lambda = [0:1:-\frac{4}{5}:\frac{3}{5}]$ is a point on the projective closure of $V$.  An eigenvector of $[[z_2]]$ associated to the eigenvalue $-\frac{4}{5}$ is $\begin{bmatrix} 1\\ -1\\ 1 \\ -2 \end{bmatrix}$, and this is also an eigenvector of $[[z_3]]$ for  $\frac{3}{5}$.
Hence the eigenvector polynomials for multiplication by $z_2$ and $z_3$ are both given by 
 \begin{eqnarray*}
 v_2(z) = \tfrac{25}{28}(z_1^2-z_1z_2+z_1z_3-2z_2^2) = 
 v_3(z), 
 \end{eqnarray*}
 where we normalize so that at $z=(1,-\frac{4}{5},\frac{3}{5})$  the polynomial evaluates to 1 (see (\ref{eqn:2.6a})).   It follows easily that the directional polynomial for $\lambda=[0:1:-\frac{4}{5}:\frac{3}{5}]$ is
 $$
 \v_{\lambda}(z)=\tfrac{25}{28}(z_1^2-z_1z_2+z_1z_3-2z_2^2).
 $$
 Another calculation gives $$(\v_{\lambda}*\v_{\lambda})(z)=\tfrac{25}{28}z_1^2(z_1^2-z_1z_2+z_1z_3-2z_2^2)-\tfrac{75}{112}z_1^2+\tfrac{100}{49}z_1z_2-\tfrac{25}{14}z_1z_3+\tfrac{1625}{392}z_2^2-\tfrac{125}{98}$$ (cf. (\ref{eqn:3.0}) below).
\end{example}

For a positive integer $s\geq\deg\v_{\lambda}:=a$, define the homogeneous polynomial  $\v_{\lambda,s}(z) := z_1^{s-a}\v_{\lambda}(z).$ 
With $\calW$ as above, we have $\calW=\{\v_{\lambda,s}: s\geq\deg(\v_{\lambda})\}$.  Equation (\ref{eqn:3.1}) also has the following useful consequence:
 \begin{equation}\label{eqn:3.0}
 \hbox{For all } z\in V,\quad (\v_{\lambda,s}(z))^q = \v_{\lambda,sq}(z) + r(z) \quad  (\deg(r)<sq).
 \end{equation}

\begin{remark} \label{rem:3.2} \rm
In (\ref{eqn:3.0}) above we mean $(\v_{\lambda,s}(z))^q = \v_{\lambda,s}(z)*\cdots*\v_{\lambda,s}(z)$ ($q$ times).   In what follows we will simplify things by writing $pq$, $p^2$, etc. for $p*q$, $p*p$, etc.  This will present no problem as we will be restricting our attention to points of $V$.  More generally, we will implicitly take normal forms of various expressions (i.e., apply $\rho_V(\cdot)$) so that we stay in $\CC[V]$. 
\end{remark}
  
  \def\calC{\mathcal{C}}
  \begin{definition}\label{def:3.2a}\rm
  Let $K\subset V$ be a compact set, and $Q\in\CC[V]$.  Let $\calC(Q)$ denote the collection of polynomials
  $$
  \calC(Q):=\{p\in\CC[V]: \exists n\in\NN, p(z) = Q(z)^{n} + r(z),\ \deg(r)<n\deg(Q)\}.  
  $$
Define 
$$\tau(K,Q,n) \ := \ \left(\inf\{\|p\|_K: p\in\calC(Q),\ \deg(p)\leq n\deg(Q) \}\right)^{\frac{1}{n\deg(Q)}},
$$
and define the $Q$-Chebyshev constant of $K$ by 
$$
\tau(K,Q) :=  \limsup_{n\to\infty}\tau(K,Q,n).
$$
  \end{definition}

  \begin{proposition} \label{prop:3.4}
  \begin{enumerate}[(1)]
\item  $\tau(K,Q)=\tau(K,\hat Q)$.\footnote{Recall $\hat Q$ is the leading homogeneous part of $Q$: $\deg(Q-\hat Q)<\deg Q$.} 
 \item $\tau(K,\alpha Q)=|\alpha|^{\frac{1}{\deg Q}}\tau(K,Q)$ for all $\alpha\in\CC$, $\alpha\neq 0$.
\item   We have $\dss \tau(K,Q) =  \lim_{n\to\infty}\tau(K,Q,n)$, i.e., the limit exists.
\end{enumerate}
  \end{proposition}
  
  \begin{proof}
  The first property follows almost immediately by definition, as  $\calC(Q)=\calC(\hat Q)$.  For the second property, a calculation (i.e., factoring out the correct power of $\alpha$) shows that $q\in\calC(\alpha Q)$ implies $q=\alpha^{\frac{\deg q}{\deg Q}}p$ for some  $p\in\calC(Q)$.  The second property then follows easily.

For the last property, it suffices to show that $\tau(K,Q)\leq\liminf_{n\to\infty}\tau(K,Q,n)$.     
  Let $\epsilon>0$ and choose $n_0\in\NN$ such that $\tau(K,Q,n_0)\leq\liminf_{n\to\infty}\tau(K,Q,n)+\epsilon.$  Given an integer $m>n_0\deg(Q)$, write $m=n_0\deg(Q)q+r$ with $0\leq r<n_0\deg(Q)$.  Take a polynomial $p_0\in\calC(Q)$ with $\deg(p_0)=n_0\deg(Q)$ such that $$\|p_0\|_K\leq(\tau(K,Q,n_0)+\epsilon)^{n_0\deg(Q)}.$$
  
   We have $p:=p_0^qQ^r$ is a polynomial in $\calC(Q)$ with $\deg(p)=m$, so
  \begin{eqnarray*}
  \tau(K,Q,m)\leq \|p_0^qQ^r\|_K^{1/m} &\leq& \|Q\|_K^{r/m}\|p_0\|_K^{n_0\deg(Q)q/m} \\
      &\leq& \|Q\|_K^{r/m} \left(\tau(K,Q,n_0)+\epsilon \right)^{n_0\deg(Q)q/m},
  \end{eqnarray*}
and hence $$\tau(K,Q,m)\leq  \|Q\|_K^{r/m}(\liminf_{n\to\infty}\tau(K,Q,n)+2\epsilon)^{n_0\deg(Q)q/m}.$$  
  Since $\frac{r}{m}\in[0,\frac{n_0\deg(Q)}{m})$ and $\frac{n_0\deg(Q)q}{m}\in(\frac{m-n_0\deg(Q)}{m},1]$, we have $\frac{r}{m}\to 0$ and $\frac{n_0\deg(Q)q}{m}\to 1$ as $m\to\infty$.  Taking the $\limsup$ as $m\to\infty$ of both sides of the above inequality then gives $$\tau(K,Q)\leq \liminf_{n\to\infty}\tau(K,Q,n)+2\epsilon.$$  Letting $\epsilon\to 0$ yields the result.
  \end{proof}
  
\begin{definition}\label{def:3.2} \rm
Let $K\subset V$ be a compact set, and $\lambda\in V\cap H_{\infty}$.  For a positive integer $s\geq \deg\v_{\lambda}$, we define the \emph{$s$-th order directional Chebyshev constant for the direction $\lambda$} by 
$$
\tau_s(K,\lambda)\ :=\ \left(\inf\{\|p\|_K: p(z)=\v_{\lambda,s}(z)+ q(z), \deg(q)<s \}\right)^{1/s}, 
$$
and define the \emph{directional Chebyshev constant for the direction $\lambda$} by 
\begin{equation}\label{eqn:cheby1} \tau(K,\lambda) \ := \ \limsup_{s\to\infty} \tau_s(K,\lambda).\end{equation}
\end{definition}

\begin{proposition} \label{prop:3.3}
We have  $\displaystyle\tau(K,\lambda)= \lim_{s\to\infty} \tau_s(K,\lambda)$, i.e., the limit of the right-hand side exists, and $\tau(K,\lambda)=\tau(K,\v_{\lambda})$.
\end{proposition}

\begin{proof}
Let $s>\deg(\v_{\lambda})=:a$ be a large positive integer.  Write $s=na+r$ where $n,r\in\NN$ and $0<r<a$.  Let $p_s(z)=\v_{\lambda,s}+q(z)$ ($\deg(q)<s$) be such that $\|p_s\|_K=\tau_s(K,\lambda)^s$.  Then  
$$z_1^{a-r}p_s=\v_{\lambda,(n+1)a}+z_1^{a-r}q=(\v_{\lambda}(z))^{n+1}+\tilde r(z), \quad \hbox{with }\deg(\tilde r)<(n+1)a,$$
where we use equation (\ref{eqn:3.0}) to get $\tilde r(z)$.  Hence \begin{equation}\label{eqn:3.2a}
\|z_1\|_K^{a-r}\tau_s(K,\lambda)^s\geq \tau(K,\v_{\lambda},n+1)^{(n+1)a}.\end{equation}

On the other hand, taking $q_n\in\calC(\v_{\lambda})$ with $\deg(q_n)=na$ such that $\|q_n\|_K=\tau(K,\v_{\lambda},n)^{na}$, we have, using equation (\ref{eqn:3.0}) again, that
$z_1^rq_n = \v_{\lambda,s}+q(z)$ where $\deg(q)<s$. A similar argument as above yields
\begin{equation}\label{eqn:3.2b}
\|z_1\|_K^r\tau(K,\v_{\lambda},n)^{na}\geq\tau_s(K,\lambda)^s.
\end{equation}
We now take $s$-th roots in (\ref{eqn:3.2a}) and (\ref{eqn:3.2b}) and let $s\to\infty$.  It is easy to see that $\frac{na}{s},\frac{(n+1)a}{s}\to 1$ and $\frac{r}{s},\frac{a-r}{s}\to 0$.  The proposition follows.\end{proof}

\begin{remark} \label{rem:4.8}  \rm
 More generally, the same proof gives $\tau(K,\lambda)=\tau(K,z_1^a\v_{\lambda})$ for any fixed non-negative integer $a$.
\end{remark}


\begin{notation}\rm
In what follows, to distinguish Chebyshev constants on different curves, we will put the curve in subscripts.  Write $\tau_{V,s}(K,Q)$ and 
$\tau_{V,s}(K,\lambda)$ to denote $s$-th order Chebyshev constants for $K$ on the curve $V$ (where $Q$ is a polynomial and $\lambda$ is a direction), and $\tau_{V}(K,Q)$ and 
$\tau_{V}(K,\lambda)$ for the respective Chebyshev constants.
\end{notation}

The next result shows how Chebyshev constants transform under linear changes of coordinates.  First, note that an invertible linear transformation $T=(T_1,...,T_N):\CC^N\to\CC^N$ (so $T_k(z)=a_{k1}z_1+\cdots+a_{kN}z_N$, $k=1,...,N$) extends to an automorphism of $\CC\PP^N=\CC^N\cup H_{\infty}$, which we also denote by $T$, with the property that $T(H_{\infty})=H_{\infty}$.  In homogeneous coordinates, 
$$T([1:z_1:\cdots:z_N])=[1:T_1(z_1,...,z_N):\cdots:T_N(z_1,...,z_N)],$$
and for points at $H_{\infty}$ of the form $[0:1:w_2:\cdots:w_N]=[0:w]$, we have
\begin{equation}\label{eqn:4.7} T([0:w])=[0:T_1(w):\cdots:T_N(w)]=   [0:1:\tfrac{T_2(w)}{T_1(w)}:\cdots:\tfrac{T_N(w)}{T_1(w)}]\end{equation}
as long as $T_1(w)\neq 0$.

\begin{proposition}\label{prop:3.8}
Let $T=(T_1,...,T_N):\CC^N\to\CC^N$ be an invertible linear transformation and suppose $T_1(w)=T_1(w_1,...,w_N)\neq 0$ whenever $[0:w_1:\cdots:w_N]\in H_{\infty}\cap V$.  Suppose in addition that both curves $V$ and $T(V)$ satisfy properties (i)--(iii) at the beginning of this section.   Then
\begin{enumerate}
 \item For any compact set $K\subset V$ and polynomial $Q\in\CC[T(V)]$,
$$
\tau_{T(V)}(T(K),Q)\ =\ \tau_{V}(K,Q\circ T).\footnote{We understand $Q\circ T\in\CC[V]$; more precisely, take $\rho_V(Q\circ T)$. (See Remark \ref{rem:3.2}.)}
$$
\item For any compact set $K\subset V$ and direction $\eta$ of $V$, 
$$
\tfrac{1}{|T_1(\eta)|}\tau_{T(V)}(T(K),T(\eta))\ = \ \tau_V(K,\eta),
$$
where $T(\eta)$ is as in equation (\ref{eqn:4.7}), using the extension of $T$ across $H_{\infty}$.
\end{enumerate}
\end{proposition}

\begin{proof}

For any polynomial $Q\in\CC[T(V)]$, a calculation shows that $\deg(Q)=\deg(Q\circ T)$.  Another calculation shows that $\calC(Q\circ T)=\calC(Q)\circ T$,\footnote{i.e., $q\in\calC(Q)$ if and only if $q\circ T\in\calC(Q\circ T)$.} and if $q\in\calC(Q)$, then $\|q\|_{T(K)}=\|q\circ T\|_K$.  The first part now follows easily from Definition \ref{def:3.2a}.

\addtocounter{footnote}{-3}

For part (2), let $\eta\in H_{\infty}\cap V$, and let $\lambda=T(\eta)$.  In what follows, $\v_{\lambda}$ will denote the directional polynomial for $\lambda$ in $\CC[T(V)]$ while $\w_{\eta}$  will denote the directional polynomial for $\eta$ in $\CC[V]$. Write $b=\deg(\v_{\lambda})$ and define
$$
 w(z):=\left(\tfrac{1}{T_1(\eta)}\right)^b\v_{\lambda}\circ T(z),  
$$
considered as a polynomial in $\CC[V]$ by taking the normal form (which we also denote by $w$).
We have $\deg(w)=b$, and evaluating at $\eta$, 
$$w(\eta)=\left(\tfrac{1}{T_1(\eta)}\right)^d\v_{\lambda}\circ T(\eta)=\v_{\lambda}(1,\tfrac{T_2(\eta)}{T_1(\eta)},...,\tfrac{T_N(\eta)}{T_1(\eta)}) = \v_{\lambda}(1,\lambda_2,...,\lambda_N) = 1.$$  Similar calculations show that $w(\tilde\eta)=0$ for any direction $\tilde\eta\neq\eta$ of $V$, and that $p(z)w(z)=\hat p(\eta)z_1^{\deg p}w(z) + r(z) $ where $\deg(r)<\deg(w)+\deg(p)$.  
Hence $w$ satisfies properties (1)--(3) of Lemma \ref{lem:3.1}, which implies that $w=z_1^a\w_{\eta}$ in $\CC[V]$ for some non-negative integer $a<b$.  We have
\begin{equation*}
\tfrac{1}{|T_1(\eta)|}\tau_{T(V)}(T(K),\v_{\lambda}) = \tau_{T(V)}(T(K),\tfrac{\v_{\lambda}}{T_1(\eta)^b}) = \tau_V(K,\tfrac{\v_{\lambda}}{T_1(\eta)^b}\circ T) = \tau_V(K,z_1^a\w_{\eta}), 
\end{equation*}
where we use Proposition \ref{prop:3.4}(2) to get the first equality, and part (1) above to get the second. Finally, applying Proposition \ref{prop:3.3} and Remark \ref{rem:4.8} yields the result.\end{proof}

\begin{remark}\rm
Chebyshev constants are invariant under translation of $V$.  Suppose $z\mapsto z+c=:\tilde z$ is a translation by some vector $c$ and $\tilde V=V+c$.  Then straightforward calculations show that $V\cap H_{\infty} = \tilde V\cap H_{\infty}$.  Also, it is easy to check that for any large positive integer $n$, the homogeneous normal forms  $\CC[V]_{=n}$ and $\CC[\tilde V]_{=n}$ are spanned by the same reduced collection of monomials of degree $n$, and for any polynomial $p$, the diagram
$$
\begin{array}{lll}
\CC[V]_{=n} & \longrightarrow & \CC[\tilde V]_{=n} \\
 & & \\
\downarrow \, _{\hsta\rho_V(p)}  & & \downarrow \, _{\hsta\rho_V(p)} \\
 & & \\
\CC[V]_{=n+\deg(p)} & \longrightarrow & \CC[\tilde V]_{=n+\deg(p)} 
\end{array}
$$
commutes, where the horizontal arrows are given by $q(z)\mapsto \widehat{\rho_{\tilde V}(q(\tilde z-c))}$, i.e., make the required change of coordinates and take the leading homogeneous part   (see Section \ref{sec:comp}  for the notation).\footnote{It is easy to see that for large $n$ we just obtain the same polynomial.}   Hence all of the algebraic computations  used to define and compute directional polynomials and Chebyshev constants on $V$ and $\tilde V$ are identical.  As a consequence, the previous proposition is also true if $T(K)$ is replaced by an affine transformation.
\end{remark}

 The next proposition shows that the study of Chebyshev constants can be restricted  to irreducible curves.

\begin{proposition} \label{prop:3.9}
Suppose $V=V_1\cup V_2$ is a union of algebraic curves, and $K\subset V$ is compact.  Let $K_1:=K\cap V_1$.  Then for $\lambda\in V_1\cap H_{\infty}$ ($\lambda=[0:1:\lambda_2:\cdots:\lambda_N]$), we have the following equalities:
\begin{equation}
\tau_{V_1}(K_1,\lambda) \ = \ \tau_V(K_1,\lambda) \ = \ \tau_V(K,\lambda).
\end{equation}
\end{proposition}

\begin{proof}
We first prove $\tau_V(K_1,\lambda)=\tau_V(K,\lambda)$.  The inequality $\tau_V(K_1,\lambda)\leq\tau_V(K,\lambda)$ follows easily from the definition, since $K_1\subset K$ and hence $\|p\|_{K_1}\leq\|p\|_K$.

We need to prove that $\tau_V(K_1,\lambda)\geq \tau_V(K,\lambda)$.   First, we fix a polynomial $g\in\I(V_2)$ such that $\|g\|_{K_1}>0$ and $\hat g(\lambda)\neq 0$.  This is possible because $\lambda\not\in V_2$ and $K_1$ contains points not in $V_2$.  


 For each positive integer $s$, let $q_s = \v_{\lambda,s}+\cdots$ be a Chebyshev polynomial of degree $s$ with $\tau_{V,s}(K_1,\lambda)^s=\|q_s\|_{K_1}$.  Consider the polynomial $q_sg$.  Then since $K=K_1\cup(K\cap V_2)$, we have
 \begin{equation} \label{eqn:3.3}
 \|q_sg\|_K\ = \ \|q_sg\|_{K_1} \leq \tau_{V,s}(K_1,\lambda)^s\|g\|_{K_1}.
 \end{equation}
 On the other hand, writing $a=\deg(g)$, we have
 $$g(z)q_s(z) =  \hat g(\lambda)\v_{\lambda,s+a}(z) +\ \cdots\ =\ \hat g(\lambda)p(z)$$
for some polynomial $p(z)=\v_{\lambda,s+a}(z)+\cdots$.  Hence 
\begin{equation}\label{eqn:3.4}
\|q_sg\|_K \ \geq \ |\hat g(\lambda)|\tau_{V,s+a}(K,\lambda)^{s+a}.
\end{equation}
Putting together (\ref{eqn:3.3}) and (\ref{eqn:3.4}) gives $|\hat g(\lambda)|\tau_{V,s+a}(K,\lambda)^{s+a}\leq \tau_{V,s}(K_1,\lambda)^s\|g\|_{K_1}$; taking $s$-th roots and letting $s\to\infty$, we obtain $\tau_V(K,\lambda)\leq\tau_V(K_1,\lambda)$. 

Altogether, $\tau_V(K_1,\lambda)=\tau_V(K,\lambda)$.

\smallskip

We now prove $\tau_{V_1}(K_1,\lambda)  =  \tau_V(K_1,\lambda)$.  Take a large positive integer $s$ and let $p_s$ be a Chebyshev polynomial of degree $s$ for $K_1$ on $V$. Then for all $z\in V$,
\begin{equation}\label{eqn:4.10}
p_s(z) \ = \  \v_{\lambda,s}(z) \ + \  r_1(z)
 \ = \  z_1^{s-a}v_1(z)v_2(z)\cdots v_N(z)  \ +\   r_1(z)
\end{equation}
where $a=\deg(v_1v_2\cdots v_N)$ and $\deg(r_1)<s$.

By Proposition \ref{prop:2.10} there are homogeneous polynomials $\varphi_1,...,\varphi_N$ such that 
\begin{equation}\label{eqn:4.12}
z_1^{a_j}v_j(z)=w_j(z)\varphi_j(z) \quad \hbox{for all } z\in V, \ j\in\{1,...,N\}.
\end{equation}
 Here the  $a_j$'s are non-negative integers, and $w_j$ denotes the eigenvector polynomial associated to the eigenvalue $\lambda_j$ of the matrix $[[z_j]]$ that represents multiplication by $z_j$ on the curve $V_1$.

Let $\varphi:=\varphi_1\varphi_2\cdots\varphi_N$ and $b:=\deg\varphi - \sum_i a_i$.  (Note that $b$ is positive since $\deg\varphi_i\geq a_i$ for all $i$.)  Let 
$\w_{\lambda}$ be the directional polynomial on $V_1$ for $\lambda$.  We have $w_1w_2\cdots w_N=\w_{\lambda,c}$ for some non-negative integer $c$.   Using (\ref{eqn:4.10}), we have for all $z\in V_1$ that 
$$
p_s(z)\ = \ z_1^{s-a}\w_{\lambda,c}(z)\varphi(z) \ + \ r_1(z)
\ = \ \varphi(\lambda) z_1^{s-a+b}\w_{\lambda,c}(z) \ + \ r_2(z),
$$
and the fact that $\v_{\lambda}(\lambda)\neq 0$ means that $\varphi(\lambda)\neq 0$. Comparing degrees also gives $c=a-b$.  We now normalize $p_s$ by setting  $q_s:=\frac{p_s}{\varphi(\lambda)}$ to obtain a competitor for a Chebyshev polynomial of degree $s$ for $K_1$ in the direction $\lambda$ on $V_1$.  Hence
$$
\tau_{V,s}(K_1,\lambda)^s \ = \ \|p_s\|_{K_1}\ = \ |\varphi(\lambda)| \|q_s\|_{K_1} \ \geq \ |\varphi(\lambda)|\tau_{V_1,s}(K_1,\lambda)^s.
$$ 
Taking $s$-th roots and letting $s\to\infty$, we have $\tau_{V}(K_1,\lambda)\geq\tau_{V_1}(K_1,\lambda)$.  

For the reverse inequality, let $q_s$ be a Chebyshev polynomial for $K_1$ of degree $s$ with $\|q_s\|_{K_1}=\tau_{V_1,s}(K_1,\lambda)^s$, and let $p_{s+b}:=q_s\varphi$.  In what follows we  assume that  $\|\varphi\|_{K_1}>0$.  We have $$p_{s+b}(z) \ = \  z_1^{s-(a-b)}\w_{\lambda}(z)\varphi(z)\ + \ r_3(z) \  = \  z_1^{s+b-a}\v_{\lambda}(z) \ + \  r_3(z),$$ so that  
$$
\tau_{V,s+b}(K_1,\lambda)^{s+b} \ \leq \ \|p_{s+b}\|_{K_1} \ \leq \ \|q_s\|_{K_1}\|\varphi\|_{K_1} \ = \ \tau_{V_1,s}(K_1,\lambda)^s\|\varphi\|_{K_1}.
$$
Taking $s$-th roots and letting $s\to\infty$ gives  $\tau_{V}(K_1,\lambda)\leq\tau_{V_1}(K_1,\lambda)$, as desired. 
If it happens that $\|\varphi\|_{K_1}=0$, then $\tau_{V}(K_1,\lambda)\leq\tau_{V_1}(K_1,\lambda)$ may be shown by replacing $\varphi$ with $\varphi+1$ in the preceding argument.

Hence $\tau_{V_1}(K_1,\lambda)  =  \tau_V(K_1,\lambda)$.
\end{proof}

\begin{remark}\rm
The directional Chebyshev constants $\tau(K,\lambda)$ were first defined in \cite{mau:chebyshev} for curves in $\CC^2$.  Propositions \ref{prop:3.3} and \ref{prop:3.9} above generalize, respectively, Theorem 4.5 and Proposition \ref{prop:3.9} in \cite{mau:chebyshev}.
\end{remark}

The following characterization of directional Chebyshev constants will be useful when studying the transfinite diameter in the next section. 
For a compact set $K\subset V$ and a direction $\lambda\in H_{\infty}\cap V$, define 
\begin{equation} \label{eqn:Ms}
t_s(K,\lambda) \, := \, \inf\Bigl\{\|p\|_K\colon p(z) = \v_{\lambda,s}(z) + \sum_{\mu\neq\lambda}a_{\mu}\v_{\mu,s}(z)  + q(z),\, a_{\mu}\in\CC,\, \deg q<s\Bigr\}^{\frac{1}{s}}\!\! .  
\end{equation}
Note that the polynomials of Definition \ref{def:3.2} have $a_{\mu}=0$.  We now verify that the same constant is obtained in the limit.

\begin{lemma} \label{lem:4.13} With $t_s(K,\lambda)$ as defined above, 
$\displaystyle\lim_{s\to\infty}t_s(K,\lambda)\ = \ \tau(K,\lambda)$.
\end{lemma}

\begin{proof}
Clearly $t_s(K,\lambda)\leq\tau_s(K,\lambda)$ for all $s$, and so   $\limsup_{s\to\infty}t_s(K,\lambda)\leq\tau(K,\lambda)$.

Next, let $a=\deg(\v_{\lambda})$.  Given a large positive integer $s>a$, take a polynomial $p_s=\v_{\lambda,s} + \sum_{\mu\neq\lambda}a_{\mu}\v_{\mu,s} \ +\ q_1(z)$ with the property that $\|p_s\|_K=t_s(K,\lambda)^s$.  Then 
\begin{eqnarray*}
p_s(z)\v_{\lambda}(z) &=& \v_{\lambda}(\lambda)\v_{\lambda,s+a}(z) + \sum_{\mu\neq\lambda}a_{\mu}\v_{\mu}(\lambda)\v_{\lambda,s}(z) \ +\ q_2(z) \\
&=& \v_{\lambda,s+a}(z)+q_2(z),
\end{eqnarray*} 
where $\deg(q_2)<s+a$.  Here we use equation (\ref{eqn:3.1}) for the first equality and parts (1) and (2) of Lemma \ref{lem:3.1} for the second.  Hence $\|p_s\v_{\lambda}\|\geq\tau_{s+a}(K,\lambda)$, and 
$$
t_s(K,\lambda)^s\|\v_{\lambda}\|_K\ \geq \ \|p_s\v_{\lambda}\|\ \geq \ \tau_{s+a}(K,\lambda)^{s+a}.
$$
Taking $s$-th roots and letting $s\to\infty$,  we have 
$\liminf_{s\to\infty}\tau_s(K,\lambda) \ \geq \tau(K,\lambda)$.  This concludes the proof. 
\end{proof}

\section{Transfinite diameter}

In this section, we study the transfinite diameter of a compact subset of an algebraic curve in $\CC^N$. 
To make use of previous results, we restrict for the moment to an algebraic curve that satisfies properties (i)--(iii) listed at the beginning of the previous section.

\smallskip \def\calC{\mathcal{C}}

Let $V\cap H_{\infty}=\{\lambda_1,...,\lambda_d\}$, where we write $\lambda_j=[0:1:\lambda_{j2}:\cdots:\lambda_{jN}]$ for each $j\in\{1,...,d\}$.  Fix a positive integer $a\geq\max_{j=1,...,d} \deg(\v_{\lambda_j})$.
Consider the following collection 
$\calC$ of polynomials: 

\smallskip

{$z^{\alpha}\in\calC$ for all $z^{\alpha}\in\calB$ with $|\alpha|< a$;\footnote{$\calB$ is as in Section \ref{subsection:groebner}.}  and}  

{$\v_{\lambda_j,s}\in\calC$ for all $j\in\{1,...,d\}$ and all $s\in\NN$ with $s\geq a.$}

\smallskip

We put an ordering $\prec$ on $\calC$ as follows.  First by degree, i.e., $\deg(p)<\deg(q)$ implies $p\prec q$.  For elements of the form $z^{\alpha}$ with $|\alpha|<a$, we use any graded ordering (e.g. grevlex).  For higher degree elements of the form $\v_{\lambda_j,s}$, we induce an ordering on $\calC$ by ordering the directions at infinity, e.g.  
\def\grvlx{\mathit{grvlx}}
\begin{eqnarray*}
\v_{\lambda_j,s}\prec \v_{\lambda_k,s} &\hbox{if}&  j<k. 
\end{eqnarray*}

For a positive integer $n$, let $\CC[V]_{\leq n}=\{p\in\CC[V]:\deg(p)\leq n\}$.\footnote{Note that in our notation, ($\CC[V]_{\leq n}\setminus\CC[V]_{\leq n-1})\supsetneq\CC[V]_{=n}$, as the latter set contains only homogeneous polynomials.} Let $m_n$ denote the dimension of this vector space, and define
$l_n:=\sum_{k=1}^n k(m_k-m_{k-1}).$

\begin{lemma}\label{lem:4.1}
For each positive integer $n$, the polynomials $\calC_n:=\{p\in\calC:\deg(p)\leq n\}$ form a basis for $\CC[V]_{\leq n}$, and hence $\calC$ is a basis for $\CC[V]$.
\end{lemma}

\begin{proof}
For $n<a$ this is trivial as $\calB$ is the monomial basis.
For $n\geq a$, note that $\CC[V]_{=n}$ has dimension $d$ (by Proposition \ref{prop:2.8}).  We verify that the set $\{\v_{\lambda_j,n}\}$ is linearly independent in $\CC[V]_{=n}$.  For any linear combination $p=\sum_{k\neq j}c_k\v_{\lambda_k,n}$, we have $p(1,\lambda_j)=0$ but $\v_{\lambda_j}(1,\lambda_j)\neq 0$.  Hence  $\{\v_{\lambda_j,n}\}_{j=1}^d$ spans $\CC[V]_{=n}$.  Assuming $\calC_{n-1}$ is a basis for  $\CC[V]_{\leq n-1}$, clearly $\calC_n=\calC_{n-1}\cup\{\v_{\lambda_j,s}\}_{j=1}^d$ spans $\CC[V]_{\leq n}$, as $d$ linearly independent elements are added, and the dimension increases by $d$. 

 The lemma now follows by induction.
\end{proof}

\begin{remark}\rm As an immediate consequence, $l_n$ is given by the sum of the degrees of all polynomials in $\calC$ of degree $\leq n$.\end{remark}

  Write $\calC$ as a sequence $\{e_j\}_{j=1}^{\infty}$ by listing the polynomials according to the ordering $\prec$ defined above (i.e., $e_1=1$, $e_j\prec e_k$ iff $j<k$).  Next, for a positive integer $n$, consider a collection of points $\{\zeta_1,...,\zeta_n\}\subset V$.  Define the Vandermonde determinant
\def\Van{\mathrm{Van}}
$$
\Van_{\calC}(\zeta_1,...,\zeta_n) \ := \ \det \begin{pmatrix}
1 & 1 & \cdots & 1 \\
e_2(\zeta_1) & e_2(\zeta_2) & \cdots & e_2(\zeta_n)\\
\vdots & \vdots & \ddots & \vdots \\
e_n(\zeta_1)  & e_n(\zeta_2) & \cdots & e_n(\zeta_n)
\end{pmatrix}.
$$

\def\calC{\mathcal{C}}
\begin{definition}
\label{def:transfd}
\rm
For a positive integer $n$, put
\begin{equation*}
V_n\ :=\ \sup\{\, |\Van_{\mathcal{C}}(\zeta_1,...,\zeta_{m_n})|:\
\{\zeta_1,...,\zeta_{m_n}\}\subset K\}.
\end{equation*}
The \emph{transfinite diameter of $K$,
$d(K)$}, is defined by
\begin{equation} \label{eqn:transfd}
d(K)\ :=\ \limsup_{n\to\infty} (V_n)^{1/l_n}.
\end{equation}
\end{definition}

The main theorem relates the transfinite diameter to the directional Chebyshev constants, and generalizes Theorem 5.7 of \cite{mau:chebyshev} to the $\CC^N$ setting.

\begin{theorem}\label{thm:d=t}
Let $K\subset V$ be a compact set.  Then the limit $\displaystyle\lim_{n\to\infty}
(V_n)^{1/l_n} \ = \ d(K)$ exists and
\begin{equation} \label{eqn:transfd=cheby}
d(K)\ =\   \biggl(\, \prod_{j=1}^d
\tau(K,\lambda_j)\, \biggr)^{1/d}.
\end{equation}
\end{theorem}

We first establish some bounds relating Chebyshev constants, specifically the constants of finite order given in Definition \ref{def:3.2} and equation (\ref{eqn:Ms}). We will need some more notation. For $n>a$ and $j=1,...,d$, set
$$
V_{n,j}=\sup\left\{\Van_{\calC}(\zeta_1,...,\zeta_{m_{n}+j}): \ \{\zeta_1,...,\zeta_{m_n+j}\}\subset K  \right\}.
$$
Note that $V_{n,d}=V_{n+1}$.
  For convenience, we also put $V_{n,0}=V_n$.
\begin{lemma}
 For $n>a$ and $j=1,...,d$, we have 
$$
t_{n+1}(K,\lambda_j)^{n+1}\leq\frac{V_{n,j}}{V_{n,j-1}} \leq (m_n+j)\tau_{n+1}(K,\lambda_j)^{n+1}.
$$
\end{lemma}

\begin{proof}
Let $\{\zeta_1,...,\zeta_{m_n+j-1}\}$ be a collection of $m_n+j-1$ points such that $V_{n,j-1}=\Van_{\calC}(\zeta_1,...,\zeta_{m_n+j-1})$.  Define the polynomial 
$$
p(z)=\frac{\Van_{\calC}(\zeta_1,...,\zeta_{m_n+j-1},z)}{\Van_{\calC}(\zeta_1,...,\zeta_{m_n+j-1})}.
$$
Then expanding the Vandermonde determinant down the last column, we have 
$$p(z) = \frac{\v_{\lambda_j,n+1}(z)\Van_{\calC}(\zeta_1,...,\zeta_{m_n+j-1}) + q(z)}{\Van_{\calC}(\zeta_1,...,\zeta_{m_n+j-1})} =\v_{\lambda_j,n+1}(z)+r(z)$$ where $\deg(q)=\deg(r)\leq n+1$. 
Hence $t_{n+1}(K,\lambda_j)^{n+1}\leq\|p\|_K\leq \frac{V_{n,j}}{V_{n,j-1}}$.  This proves the lower inequality.

Next, let now $\{\zeta_1,...,\zeta_{m_n+j}\}$ be a collection of $m_n+j$ points such that $V_{n,j}=\Van_{\calC}(\zeta_1,...,\zeta_{m_n+j})$.  Now let $p(z)=\v_{\lambda_j,n+1}(z) + q(z)$ (with $\deg(q)\leq n$) be a polynomial such that $\|p\|_K=\tau_{n+1}(K,\lambda)^{n+1}$. Then  
$$
\det\begin{bmatrix}
 1&\cdots& 1 \\
\vdots & \ddots & \vdots \\
\v_{\lambda_j,n+1}(\zeta_1)&\cdots & \v_{\lambda_j,n+1}(\zeta_{m_n+j})
                                              \end{bmatrix}
= \det\begin{bmatrix}
 1&\cdots& 1 \\
\vdots & \ddots & \vdots \\
p(\zeta_1)&\cdots & p(\zeta_{m_n+j})
                                              \end{bmatrix},
$$
since replacing $\v_{\lambda_j,n+1}$ with $p=\v_{\lambda_j,n+1}+q$ in the last row is the same as adding to this row a linear combination of previous rows (given by the coefficients of $q$).  Expanding the determinant along the last row and taking absolute values yields
\begin{eqnarray*}
V_{n,j} &\leq& \sum_{s=1}^{m_n+j} |\Van_{\calC}(\zeta_1,...,\hat\zeta_s,...,\zeta_{m_n+j})|\cdot|p(\zeta_s)| \\ 
&\leq& (m_n+j)V_{n,j-1}\|p\|_K  =  (m_n+j)V_{n,j-1}(\tau_{n+1}(K,\lambda))^{n+1}, 
\end{eqnarray*}
where $\hat\zeta_s$ indicates that $\zeta_s$ is omitted.  This proves the upper inequality.
\end{proof}

\begin{corollary}
For $n>a$, we have
$$
\prod_{j=1}^d t_{n+1}(K,\lambda_j)^{n+1} \leq   \frac{V_{n+1}}{V_{n}} \leq \frac{m_{n+1}!}{m_n!}\prod_{j=1}^d\tau_{n+1}(K,\lambda_j)^{n+1}.
$$
\end{corollary}

\begin{proof}
We have $$\frac{V_{n+1}}{V_n} = \frac{V_{n,d}}{V_{n,d-1}}\frac{V_{n,d-1}}{V_{n,d-2}}\cdots\frac{V_{n,1}}{V_{n,0}}.$$  Now apply, to each quotient on the right-hand side, the upper and lower bounds in the previous result.
\end{proof}

\begin{proof}[Proof of Theorem \ref{thm:d=t}]
Let $\epsilon>0$. By Proposition \ref{prop:3.3} and Lemma \ref{lem:4.13}, we have for each $j$ that  
$$\lim_{n\to\infty} t_{n}(K,\lambda_j)= \tau(K,\lambda_j) \hbox{ and } \lim_{n\to\infty} \tau_{n}(K,\lambda_j)= \tau(K,\lambda_j).$$  Hence there exists  an integer $n_0>a$ sufficiently large such that for all $j=1,...,d$, we have 
\begin{equation}\label{eqn:4.3}
t_{n_0}(K,\lambda_j)\geq \tau(K,\lambda_j)-\epsilon \ \hbox{ and } \  \tau_{n_0}(K,\lambda_j)\leq\tau(K,\lambda_j)+\epsilon \ \hbox{ for all }n\geq n_0.
\end{equation}
 For $n>n_0$, write 
$$
V_n = \frac{V_n}{V_{n-1}}\frac{V_{n-1}}{V_{n-2}}\cdots\frac{V_{n_0+1}}{V_{n_0}}V_{n_0};
$$
applying the previous corollary to the product on the right-hand side of the above equation, we obtain
$$
V_{n_0}\prod_{s=n_0+1}^n \prod_{j=1}^d t_{s}(K,\lambda_j)^{s} \leq V_n \leq V_{n_0}\frac{m_{n}!}{m_{n_0}!}\prod_{s=n_0+1}^n \prod_{j=1}^d \tau_{s}(K,\lambda_j)^{s},
$$
and by (\ref{eqn:4.3}), this becomes
$$
V_{n_0}\prod_{j=1}^d (\tau(K,\lambda_j)-\epsilon)^{A_n} \leq V_n \leq
V_{n_0}\frac{m_{n}!}{m_{n_0}!}\prod_{j=1}^d (\tau(K,\lambda_j)+\epsilon)^{A_n}
$$
where $A_n=\sum_{s=n_0+1}^ns = \frac{(n-n_0)(n+n_0+1)}{2}$.  Taking $l_n$-th roots, we have for all $n>n_0$ that
$$
V_{n_0}^{1/l_n}\prod_{j=1}^d (\tau(K,\lambda_j)-\epsilon)^{A_n/l_n} \leq V_n^{1/l_n} \leq
V_{n_0}^{1/l_n}\left(\frac{m_{n}!}{m_{n_0}!}\right)^{1/l_n}\prod_{j=1}^d (\tau(K,\lambda_j)+\epsilon)^{A_n/l_n}.
$$

We want to take the limit as $n\to\infty$.  Note that  $m_n-m_{n-1}=d$ for $n>n_0$; hence $m_n=dn+c$ for some integer $c$.  Also, we have 
$$l_n=l_{n_0}+ \sum_{s=n_0+1}^nsd = l_{n_0} + \frac{d(n-n_0)(n+n_0+1)}{2},$$ 
which implies that $\lim_{n\to\infty}\frac{A_n}{l_n}=\frac{1}{d}$.   Clearly $V_{n_0}^{1/l_n},(m_{n_0}!)^{1/l_n}\to 1$ as $n\to\infty$.    Since $l_n$ is of order $n^2$ and $m_n$ is of order $n$, there is a constant $b>0$ such that $b\cdot nm_n\leq l_n$ for all $n>n_0$, and so  
$$1\leq (m_n!)^{1/l_n}\leq m_n^{m_n/l_n} \leq (dn+c)^{\frac{1}{bn}}\longrightarrow 1\quad \hbox{as } n\to\infty.$$
Hence 
$$
\left(\prod_{j=1}^d (\tau(K,\lambda_j)-\epsilon)\right)^{\frac{1}{d}} \leq \liminf_{n\to\infty} V_n^{\frac{1}{l_n}} \leq \limsup_{n\to\infty} V_n^{\frac{1}{l_n}} \leq
\left(\prod_{j=1}^d (\tau(K,\lambda_j)+\epsilon)\right)^{\frac{1}{d}}.
$$
Finally, let $\epsilon\to 0$.  Then (\ref{eqn:transfd=cheby}) follows, which completes the proof.
\end{proof}

From an algebraic point of view, it is natural to define transfinite diameter using monomials (as in the Introduction).  

Let  $K\subset V\subset\CC^N\subset\CC\PP^N $ be a compact set, and let $I=\I(V)$ be the ideal of $V$.  Consider the monomials $\{z^{\alpha}: z^{\alpha}\not\in\langle\lt(I)\rangle\}$ (which form a basis for $\CC[V]$) listed according to grevlex order as a sequence  $\{z^{\alpha_j}\}_{j=1}^{\infty}$.  Define  
$$
\Van(\zeta_1,...,\zeta_n):= \det\begin{pmatrix}
1 & 1 & \cdots & 1 \\
z^{\alpha_2}(\zeta_1) & z^{\alpha_2}(\zeta_2) & \cdots & z^{\alpha_2}(\zeta_n) \\
\vdots &\vdots & \ddots & \vdots \\
z^{\alpha_n}(\zeta_1) & z^{\alpha_n}(\zeta_2) & \cdots & z^{\alpha_n}(\zeta_n)
\end{pmatrix}.
$$

\begin{corollary} \label{cor:4.6} For a positive integer $n$, put
\begin{equation*}
\tilde V_n\ :=\ \sup\{\, |\Van(\zeta_1,...,\zeta_{m_n})|:\
\{\zeta_1,...,\zeta_{m_n}\}\subset K\}.
\end{equation*}
Then $\lim_{n\to\infty} \tilde V_n^{1/l_n} = d(K)$.
\end{corollary}

\addtocounter{footnote}{-2}

The corollary says that $d(K)$ may be given by Definition \ref{def:transfd} with $\Van_{\calC}$ replaced by $\Van$.  That the two limits are equal can be seen as follows.  For large $n$, a basis of $\CC[V]_{=n}$ is given alternatively by $\{z^{\alpha_j}\}_{j=m_{n-1}+1}^{m_n}$ and $\{\v_{\lambda_k,n}\}_{k=1}^d$.  One can therefore use row operations to transform the rows in positions $m_{n-1}+1,...,m_n$ of the matrix for $\Van_{\calC}$ (there are $d$ of these) into the corresponding rows of the matrix for $\Van$.  Note that the bases of $\CC[V]_{=n}$ are the same for each $n$ up to a power of $z_1$.   Precisely, for a large positive integer $n_0$ and $k\in\{1,...,d\}$, we have 
$
\v_{\lambda_k,n}(z) = z_1^{n-n_0}\v_{\lambda_k,n_0}(z)
$
and $z^{\alpha_{j}+dn}=z_1^{n-n_0}z^{\alpha_{j}+dn_0}$ 
for all $n>n_0$ (cf. Proposition \ref{prop:2.5}).  Hence the \emph{same} row operations work for each block of rows corresponding to a fixed degree.

It follows from properties of determinants that
\begin{equation} \label{eqn:5.4}
\Van_{\calC}(\zeta_1,...,\zeta_{m_n}) = CR^{n-n_0}\Van(\zeta_1,...,\zeta_{m_n})
\end{equation}
where $C=C(n_0)$ is a fixed constant and $R$ is the factor obtained each time we use the row operations on the rows of $\Van_{\calC}$ corresponding to $\{\v_{\lambda_k,s}\}_{k=1}^d$ to get the rows of $\Van$ corresponding to the monomials
$\{z^{\alpha_j}\}_{j=m_{s-1}+1}^{m_s}$ for each $s=n_0+1,...,n$.  Clearly $R\neq 0$.

The important point is that $l_n$ is quadratic in $n$ and therefore $(CR^{n-n_0})^{1/l_n}\to 1$  as $n\to\infty$.  Hence the $l_n$-th roots of the Vandermonde determinants in (\ref{eqn:5.4}) are almost equal for large $n$.

A precise proof along the above lines was given as Corollary 5.14 of \cite{mau:chebyshev}; although only curves in $\CC^2$ (or $\CC\PP^2$) were  considered there, the argument is general.

\medskip

The following properties of transfinite diameter follow immediately from properties of directional Chebyshev constants (Propositions \ref{prop:3.8} and \ref{prop:3.9}).

\begin{corollary} 
\begin{enumerate}
\item Let $\lambda_1,...,\lambda_d$ be the directions of $V$, and  $T=(T_1,...,T_N):\CC^N\to\CC^N$ be a linear transformation such that $T_1(\lambda_j)\neq 0$ for all $j=1,...,d$.  Suppose $V$ and $T(V)$ satisfy properties (i)--(iii) at the beginning of Section \ref{sec:cheby}. Then for any compact set $K\subset V$, 
\begin{equation} \label{eqn:5.4a}
d_V(K)\prod_{j=1}^d T_1(\lambda_j)\ =\ d_{T(V)}(T(K)).
\end{equation}
\item Let $V=V_1\cup V_2$ where $V_1, V_2$ are curves  of degrees $d_1$ and $d_2$ respectively, and satisfy properties (i)--(iii).   Then writing $d=d_1+d_2$, we have 
$$
d_V(K) = d_{V_1}(K)^{\frac{d_1}{d}}d_{V_2}(K)^{\frac{d_2}{d}}. \eqno{\qed}
$$
\end{enumerate} 
\end{corollary}

By the first part of the above corollary, $d_V(K)=0$ if and only if $d_{T(V)}(T(K))=0$. Also, ratios of transfinite diameters are invariant under linear changes of coordinates (as long as all quantities are defined and the ratio makes sense) since the extra factors on the left-hand side of (\ref{eqn:5.4a}) are independent of the set. Given compact sets $K_1,K_2\subset V$, with $d_V(K_2)>0$, we have 
 \begin{equation}\label{eqn:4.4} \frac{d_V(K_1)}{d_V(K_2)}\ =\ \frac{d_{T(V)}(T(K_1))}{d_{T(V)}(T(K_2))}=: d_{V}(K_1,K_2) .\end{equation}

\begin{definition}
Given $K_1,K_2\subset V$ with $d_V(K_2)>0$, define the transfinite diameter of $K_1$ relative to $K_2$ to be $d_{V}(K_1,K_2)$.  
\end{definition}

One can therefore normalize transfinite diameter by computing it relative to some fixed set (i.e. fix $K_2$ in (\ref{eqn:4.4})), to obtain an intrinsic notion independent of coordinates.  

\begin{example}\rm
For the complex line $V=\{z\in\CC^N: z_1=z_2=\cdots=z_N\}$, let $d_V(\cdot)$ be the transfinite diameter on $V$ (of Theorem \ref{thm:d=t} or Corollary \ref{cor:4.6}).  Set   $\delta(K):=d_V(K,D)$ where $D = \{z\in V: |z_1|^2+|z_2|^2+\cdots+|z_N|\leq 1\}$ is the ``unit disk''  in $V$.  Then $\delta(K)$ coincides with the classical transfinite diameter of $K$ in the plane (defined in terms of the restriction to $V$ of the usual metric in $\CC^N$).
\end{example}

Fixing a normalization, we can extend the notion of relative transfinite diameter to any algebraic curve $V$ with the property that $V\cap H_{\infty}$ is a transverse intersection of nonsingular points.  It can be explicitly  computed by changing, if necessary, to ``good'' coordinates, i.e., such that the image of $V$ under this change of coordinates satisfies properties (i)--(iii) at the beginning of Section \ref{sec:cheby}.

\begin{example}\rm
Consider the curve $V$ in $\CC^2$ given by the equation $z_1z_2=1$.  Then $V$ is a transverse intersection of nonsingular points, but its coordinates are ``bad'' since $H_{\infty}$ contains the point $[0:0:1]$.  Letting $T(z_1,z_2)=(z_1+z_2,z_1-z_2)$, we have that $T(V)$ is the curve given by $z_1^2-z_2^2=4$, in which our theory applies and we can compute $d_{T(V)}(K)$ for a compact $K\subset T(V)$.  Let $D=\{(e^{i\theta},e^{-i\theta}):\theta\in\RR\}\subset V$; then $T(D)=\{(2\cos\theta,2i\sin\theta):\theta\in\RR\}$.  We define $d_V(K,D):=d_{T(V)}(T(K),T(D))$.
\end{example}

In general, suppose $V$ has ``bad'' coordinates.  Let $T,S:\CC^N\to\CC^N$ be invertible linear maps that provide ``good'' coordinates for $V$.  Fix a compact set $D$.  Then using equation (\ref{eqn:4.4}) applied to $T(V)$ and $S\circ T^{-1}$, we can see that for any compact set $K\subset V$, we have 
$$d_{T(V)}(T(K),T(D))=d_{S(V)}(S(K),S(D))$$ as long as e.g. $d_{T(V)}(T(D))\neq 0$.  Hence $d_V(K,D)$ is a well-defined quantity since its value is independent of which ``good'' coordinates are chosen for the computation.


\begin{remark} \rm
Note that a linear change of coordinates does not work for the curve $V\subset\CC^2$ given by $z_2=z_1^2$, since $[0:0:1]$ is not a transverse intersection of  $V$ with $H_{\infty}$.  In order to manage such a case it seems that one would have to deal with multiple eigenvalues of multiplication matrices.
\end{remark}

\section{Concluding remarks: pluripotential theory}

\def\calL{\mathcal{L}}
Let $\calL$ be the class of global plurisubharmonic (psh) functions on $\CC^N$ of logarithmic growth,  i.e., 
$$
\calL = \{u \hbox{ psh on }\CC^N:\ \exists C\in\RR \hbox{ such that } u(z)\leq \log^+|z|+C, \ \forall z\in\CC^N\}.
$$
Given a compact subset $K$ of $\CC^N$, define 
$$
V_K(z):=\sup\{u(z): u\in\calL, u\leq 0 \hbox{ on } K\}.
$$
We will call $V_K$ the \emph{Siciak-Zaharjuta extremal function} associated to $K$. Some authors define the Siciak-Zaharjuta extremal function to be the upper semicontinuous regularization $V_K^*(z):=\limsup_{t\to z} V_K(t)$.  For convenience we will call $V_K$ the \emph{unregularized} extremal function and  $V_K^*$ the \emph{regularized} extremal function.

A well-known formula of Zaharjuta and Siciak (see e.g. Chapter 5 of \cite{klimek:pluripotential}) is that
$$
V_K(z) = \sup\bigl\{\tfrac{1}{\deg(p)}\log|p(z)|:\ p \hbox{ is a polynomial with } \|p\|_K\leq 1\bigr\}.
$$

It is also well-known that either
 $V_K^*\in\calL$ or $V_K^*\equiv+\infty$, with the latter case occuring if and only if $K$ is pluripolar in $\CC^N$. In particular, this holds when $K\subset A$, where $A$ is an analytic set in $\CC^N$, i.e., for all $z\in A$ there is a neighborhood $D$ of $z$ and holomorphic functions $f_1,...,f_m$ on $D$ with 
$$
D\cap A = \{w\in\CC^N:\ f_1(w)=\cdots=f_m(w)=0\}.
$$

Sadullaev showed in \cite{sadullaev:estimate} that the unregularized extremal function $V_K$ provides a pluripotential theoretic criterion for an analytic set to be algebraic.  

\begin{theorem}
Let $A\subset\CC^N$ be an analytic set.  Suppose there exists a compact subset $K\subset A$ such that $V_K$ is locally bounded on $A$.   Then $A$ is contained in an algebraic set. \qed
\end{theorem}



Consider a compact set $K\subset V\subset\CC^N$, where $V$ is now an algebraic curve, with the property that $V_K$ is locally bounded on $V$.  Let $V_{\reg}$ denote the regular points of $V$, i.e., the points at which $V$ is locally a smooth manifold. Note that $V\setminus V_{\reg}$ is a finite set. Sadullaev has also verified the following.
\begin{theorem}
Let $K\subset V$ be a compact set such that $V_K$ is locally bounded on $V$.  Then $V_K$ is harmonic on $V_{\reg}\setminus K$ and has logarithmic growth. \qed
\end{theorem}

Suppose the curve $V$ satisfies properties (i)--(iii) at the beginning of Section \ref{sec:cheby}, with directions  $\lambda_j=[0:1:\lambda_{j2}:\cdots:\lambda_{jN}]$, $j=1,...,d$.  Since $V_K$ is of logarithmic growth on $V$, we may define for each $j$ the \emph{$j$-th directional Robin constant} of $K$ by 
$$
\rho_K(\lambda_j)  := 
\limsup_{\substack{|z|\to\infty,\, z\in V\\ [1:z]\to\lambda_j}}V_K(z)-\log|z_1|,
$$
where we write $[1:z]=[1:z_1:\cdots:z_N]=[\frac{1}{z_1}:1:\frac{z_2}{z_1}:\cdots:\frac{z_N}{z_1}]$.  
We claim that for each $j=1,...,d$,
$$
e^{-\rho_K(\lambda_j)} = \tau(K,\lambda_j).
$$
When $V$ is a complex line (i.e. $d=1$) it is straightforward to show that the quantities in the above equation may be computed in terms of the variable $z_1$ only, reducing it to a classical relation between Robin and Chebyshev constants in the plane.  The generalization to curves of higher degree will be proved in another paper.

\bibliographystyle{abbrv}
\bibliography{myreferences}

\end{document}